\documentclass[12pt,british,a4paper,reqno]{amsart}
    \usepackage[T1]{fontenc}
\usepackage[utf8]{inputenc}
	\usepackage[british]{babel}
    \DeclareMathSizes{12}{12}{7}{6}
\usepackage{wasysym}
	\usepackage{amsmath}
	\usepackage{amssymb}
	\usepackage{amsthm}

	\usepackage{adjustbox}
	\usepackage{bbm}
	\usepackage{braket}
	\usepackage{xcolor}
	\usepackage[shortlabels]{enumitem}
	\usepackage{etoolbox}
	\usepackage{fancyhdr}
	\usepackage{geometry}
	\usepackage{graphicx}
	\usepackage{ifsym}
	\usepackage{indentfirst}
	\usepackage{mathrsfs}
	\usepackage{mathtools}
	\usepackage{oplotsymbl} 
	\usepackage{proof}
	\usepackage{qtree}
	\usepackage{setspace}
	\usepackage{soul}
	\usepackage{stmaryrd}
	\usepackage{tensor}
	\usepackage{tikz}
	\usepackage{tikz-cd}
	\usepackage[normalem]{ulem}
	\usepackage{cancel}
	\usepackage{scalerel}
	\usepackage{stackengine}

    \usepackage[hidelinks, hyperfootnotes=false]{hyperref}
		\usepackage[capitalise]{cleveref}
		\usepackage{bookmark}
  \usepackage{nicematrix}
		
\makeatletter
\if@cref@capitalise
\crefname{theorem}{TheoUpperCase}{TheoUppercasesS}
\else
\crefname{theorem}{theoLowercase}{theoLowercaseS}
\fi
\makeatother
		
\DeclareFontFamily{U}{rcjhbltx}{}
\DeclareFontShape{U}{rcjhbltx}{m}{n}{<->rcjhbltx}{}
\DeclareSymbolFont{hebrewletters}{U}{rcjhbltx}{m}{n}
\let\beth\relax
\let\daleth\relax
\let\shin\relax

\DeclareMathSymbol{\beth}{\mathord}{hebrewletters}{98}
\DeclareMathSymbol{\daleth}{\mathord}{hebrewletters}{100}

\DeclareMathSymbol{\mem}{\mathord}{hebrewletters}{109}
\DeclareMathSymbol{\tsadi}{\mathord}{hebrewletters}{118}
\DeclareMathSymbol{\shin}{\mathord}{hebrewletters}{152}
\newcommand{\Beth}{\mathbin{\scalebox{1.3}{$\beth$}}}
\newcommand{\Daleth}{\mathbin{\scalebox{1.3}{$\daleth$}}}
\newcommand{\Mem}{\mathbin{\scalebox{1.3}{$\mem$}}}
\newcommand{\Tsadi}{\mathbin{\scalebox{1.3}{$\tsadi$}}}
\newcommand{\Shin}{\mathbin{\scalebox{1.3}{$\shin$}}}

\newcommand{\updave}{\textup{\davidsstar}}

\newcommand{\exactcomplete}{\heartsuit}
\newcommand{\nexangcomplete}{\clubsuit}

\newcommand{\exactWIC}{\diamondsuit}
\newcommand{\nexangWIC}{\spadesuit}
 
	\bookmarksetup{numbered,open}
	\setlist[enumerate]{itemsep=3pt,topsep=3pt,
    labelsep=5pt, 
    leftmargin=29.40003pt,
	}
    
	\renewcommand{\geq}{\geqslant}
	
	\renewcommand{\phi}{\varphi}

	\providecommand{\corollaryname}{Corollary}
	\providecommand{\definitionname}{Definition}
	\providecommand{\notationname}{Notation}
	\providecommand{\examplename}{Example}
	\providecommand{\lemmaname}{Lemma}
	\providecommand{\propositionname}{Proposition}
	\providecommand{\remarkname}{Remark}
	\providecommand{\theoremname}{Theorem}
	\providecommand{\setupname}{Setup}
	\providecommand{\conjecturename}{Conjecture}
	\providecommand{\questionname}{Question}
	\providecommand{\warningname}{Warning}

	\theoremstyle{plain}
		\newtheorem{thm}{\protect\theoremname}[section] 
		\newtheorem{prop}[thm]{\protect\propositionname}
		    \crefalias{prop}{proposition}
		\newtheorem{lem}[thm]{\protect\lemmaname}
		\newtheorem{cor}[thm]{\protect\corollaryname}
		
		\newtheorem{thmx}{\protect\theoremname}
		\newtheorem{propx}[thmx]{\protect\propositionname}

	\theoremstyle{definition}
		\newtheorem{defn}[thm]{\protect\definitionname}
		\newtheorem{notn}[thm]{\protect\notationname}
		    \crefname{notn}{notation}{notations}
		    \Crefname{notn}{Notation}{Notations}
		\newtheorem{example}[thm]{\protect\examplename}
		\newtheorem{setup}[thm]{\setupname}
		    \crefname{setup}{setup}{setups}
		    \Crefname{setup}{Setup}{Setups}

	\theoremstyle{remark}
		\newtheorem{rem}[thm]{\protect\remarkname}
		\newtheorem{warning}[thm]{\protect\warningname}
		
	\numberwithin{figure}{section}
	\numberwithin{equation}{section}

	\usetikzlibrary{matrix,arrows,decorations.pathmorphing,positioning,decorations.pathreplacing}
	\tikzset{commutative diagrams/.cd, 
		mysymbol/.style = {start anchor=center, end anchor = center, draw = none}}
    \tikzset{
    labl/.style={anchor=north, rotate=90, inner sep=1mm}
    }

	\tikzcdset{every label/.append style = {font = \footnotesize}}
	\tikzcdset{scale cd/.style={every label/.append style={scale=#1},
    cells={nodes={scale=#1}}}}

	\let\amph=& 
	
	\tikzcdset{every label/.append style = {font = \footnotesize}}

	\newcommand{\BE}{\mathbb{E}}
			\renewcommand{\BE}{\mathchoice
  			{\mathbb{E}}
  			{\mathbb{E}}
  			{\raisebox{-0.2pt}{\scalebox{.6}{$\mathbb{E}$}}}
  			{\scalebox{.5}{$\mathbb{E}$}}
		}
	\newcommand{\BF}{\mathbb{F}}
			\renewcommand{\BF}{\mathchoice
  			{\mathbb{F}}
  			{\mathbb{F}}
  			{\raisebox{-0.2pt}{\scalebox{.6}{$\mathbb{F}$}}}
  			{\scalebox{.5}{$\mathbb{F}$}}
		}
	\newcommand{\BG}{\mathbb{G}}
			\renewcommand{\BG}{\mathchoice
  			{\mathbb{G}}
  			{\mathbb{G}}
  			{\raisebox{-0.2pt}{\scalebox{.6}{$\mathbb{G}$}}}
  			{\scalebox{.5}{$\mathbb{G}$}}
		}

	\newcommand{\CC}{\mathcal{C}}
	\newcommand{\CD}{\mathcal{D}}
	\newcommand{\CE}{\mathcal{E}}

	\newcommand{\CM}{\mathcal{M}}

	\newcommand{\CP}{\mathcal{P}}

	\newcommand{\CX}{\mathcal{X}}
	\newcommand{\CY}{\mathcal{Y}}

	\newcommand{\SE}{\mathscr{E}}
	\newcommand{\SF}{\mathscr{F}}
	\newcommand{\SG}{\mathscr{G}}
	\newcommand{\SH}{\mathscr{H}}
	\newcommand{\SI}{\mathscr{I}}
	
	\newcommand{\SK}{\mathscr{K}}
    \newcommand{\SL}{\mathscr{L}}
    \newcommand{\SM}{\mathscr{M}}

		\newcommand{\Ab}{\operatorname{\mathsf{Ab}}\nolimits}

		\newcommand{\Exactcat}{\operatorname{\mathsf{Exact}}\nolimits}
		\newcommand{\exactcat}{\operatorname{\mathsf{exact}}\nolimits}
		\newcommand{\ICExactcat}{\operatorname{\mathsf{IC-Exact}}\nolimits}
		\newcommand{\ICexactcat}{\operatorname{\mathsf{IC-exact}}\nolimits}
            \newcommand{\WICExactcat}{\operatorname{\mathsf{WIC-Exact}}\nolimits}
		\newcommand{\WICexactcat}{\operatorname{\mathsf{WIC-exact}}\nolimits}

		\newcommand{\op}{\raisebox{0.5pt}{\scalebox{0.6}{\textup{op}}}}
		\newcommand{\ring}{R}

		\newcommand{\Hom}{\operatorname{Hom}\nolimits}
		\newcommand{\End}{\operatorname{End}\nolimits}

		\makeatletter
            \newcommand{\xmapsfrom}[2][]{%
            \ext@arrow3095\leftarrowfill@{#1}{#2}\mapsfromchar
            }
        \makeatother
		
		\newcommand{\cok}{\operatorname{coker}\nolimits}

	    \newcommand\restr[2]{{\left.\kern-\nulldelimiterspace#1
						\right|_{#2}}}
		\newcommand{\inn}{\mathchoice
  			{\raisebox{0.6pt}{$\hspace{3pt}\in\hspace{3pt}$}}
  			{\raisebox{0.6pt}{$\hspace{3pt}\in\hspace{3pt}$}}
  			{\raisebox{0.5pt}{\scalebox{.5}{\hspace{1pt}$\in$\hspace{1pt}}}}
  			{\scalebox{.5}{$\in$}}
		}
		
		\newcommand{\catext}[2]{\mathchoice
  			{{#1}\operatorname{-Ext}({#2})}
  			{{#1}\operatorname{-Ext}({#2})}
  			{\raisebox{-0.2pt}{\scalebox{.6}{${#1}\operatorname{-Ext}({#2})$}}}
  			{\scalebox{.5}{${#1}\operatorname{-Ext}({#2})$}}
		}

		\newcommand{\fs}{\mathfrak{s}}
		\newcommand{\ft}{\mathfrak{t}}
            \newcommand{\fu}{\mathfrak{u}}
		
		\newcommand{\com}{\mathsf{C}}
		
		\newcommand{\Exang}[1]{{#1}\operatorname{-\,\mathsf{Exang}}\nolimits}
		\newcommand{\exang}[1]{{#1}\operatorname{-\,\mathsf{exang}}\nolimits}
		\newcommand{\ICExang}[1]{\operatorname{\mathsf{IC-}}{#1}\operatorname{-\,\mathsf{Exang}}\nolimits}
		\newcommand{\ICexang}[1]{\operatorname{\mathsf{IC-}}{#1}\operatorname{-\,\mathsf{exang}}\nolimits}
            \newcommand{\WICExang}[1]{\operatorname{\mathsf{WIC-}}{#1}\operatorname{-\,\mathsf{Exang}}\nolimits}

		\let\amsamp=&
		
	\newcommand{\deff}{\coloneqq}
	
	\newcommand{\lan}{\langle}
	\newcommand{\ran}{\rangle}

		\newcommand\rwt[1]{\ThisStyle{%
  \setbox0=\hbox{$\SavedStyle#1$}%
  \stackengine{.2\LMpt}{$\SavedStyle#1$}{%
    \stretchto{\scaleto{\SavedStyle\mkern-.3mu\sim}{.5\wd0}}{.3\ht0}%
  }{O}{c}{F}{T}{S}%
}}
\DeclareSymbolFont{yhlargesymbols}{OMX}{yhex}{m}{n} \DeclareMathAccent{\rwh}{\mathord}{yhlargesymbols}{"62}

\newcommand{\id}[1]{\tensor*[]{\mathrm{id}}{_{#1}}}

\newcommand{\zerou}[1]{\tensor*[]{0}{^{#1}}}

\newcommand{\zerod}[1]{\tensor*[]{0}{_{#1}}}

\newcommand{\ad}[1]{\tensor*[]{a}{_{#1}}}

\newcommand{\cu}[1]{\tensor*[]{c}{^{#1}}}

\newcommand{\du}[1]{\tensor*[]{d}{^{#1}}}

\newcommand{\dd}[1]{\tensor*[]{d}{_{#1}}}

\newcommand{\ed}[1]{\tensor*[]{e}{_{#1}}}

\newcommand{\fd}[1]{\tensor*[]{f}{_{#1}}}

\newcommand{\Mat}[1]{\tensor*[]{\mathrm{Mat}}{_{#1}}}

\newcommand{\rd}[1]{\tensor*[]{r}{_{#1}}}

\renewcommand{\ss}[1]{\tensor*[]{s}{_{#1}}}

\newcommand{\vu}[1]{\tensor*[]{v}{^{#1}}}

\newcommand{\CCu}[1]{\tensor*[]{\CC}{^{#1}}}

\newcommand{\CDu}[1]{\tensor*[]{\CD}{^{#1}}}

\newcommand{\rwtCCu}[1]{\tensor*[]{\rwt{\CC}}{^{#1}}}

\newcommand{\rwhCCu}[1]{\tensor*[]{\rwh{\CC}}{^{#1}}}

\newcommand{\SLd}[1]{\tensor*[]{\SL}{_{#1}}}

\newcommand{\SKd}[1]{\tensor*[]{\SK}{_{#1}}}

\newcommand{\SId}[1]{\tensor*[]{\SI}{_{#1}}}

\newcommand{\SFd}[1]{\tensor*[]{\SF}{_{#1}}}

\newcommand{\SEd}[1]{\tensor*[]{\SE}{_{#1}}}

\newcommand{\rwtSFd}[1]{\tensor*[]{\rwt{\SF}}{_{#1}}}

\newcommand{\rwhSFd}[1]{\tensor*[]{\rwh{\SF}}{_{#1}}}

\newcommand{\SFu}[1]{\tensor*[]{\SF}{^{#1}}}

\newcommand{\Xd}[1]{\tensor*[]{X}{_{#1}}}

\newcommand{\rwhXd}[1]{\tensor*[]{\rwh{X}}{_{#1}}}

\newcommand{\Xcom}{\tensor*[]{X}{_{\bullet}}}

\newcommand{\Yd}[1]{\tensor*[]{Y}{_{#1}}}

\newcommand{\Endd}[1]{\tensor*[]{\End}{_{#1}}}

\newcommand{\Bethd}[1]{\tensor*[]{\Beth}{_{#1}}}

\newcommand{\rwtBethd}[1]{\tensor*[]{\rwt{\Beth}}{_{#1}}}

\newcommand{\Dalethd}[1]{\tensor*[]{\Daleth}{_{#1}}}

\newcommand{\Shind}[1]{\tensor*[]{\Shin}{_{#1}}}

\newcommand{\Tsadid}[1]{\tensor*[]{\Tsadi}{_{#1}}}

\newcommand{\Memd}[1]{\tensor*[]{\Mem}{_{#1}}}

\newcommand{\Shindp}[1]{\tensor*[]{\Shin}{_{#1}^{\,\prime}}}

\newcommand{\Tsadidp}[1]{\tensor*[]{\Tsadi}{_{#1}^{\,\prime}}}

\newcommand{\CXd}[1]{\tensor*[]{\CX}{_{#1}}}

\newcommand{\rwhCXd}[1]{\tensor*[]{\rwh{\CX}}{_{#1}}}

\newcommand{\comd}[1]{\tensor*[]{\com}{_{#1}}}

\newcommand{\exactcompletecell}[1]{\tensor*[]{\exactcomplete}{_{#1}}}

\newcommand{\nexangcompletecell}[1]{\tensor*[]{\nexangcomplete}{_{#1}}}

\newcommand{\Exangcell}[2]{\tensor*[]{\Exang{#1}}{_{#2}}}

\newcommand{\ICExangcell}[2]{\tensor*[]{\ICExang{#1}}{_{#2}}}

\newcommand{\WICExangcell}[2]{\tensor*[]{\WICExang{#1}}{_{#2}}}

\newcommand{\davecell}[1]{\tensor*[]{\updave}{_{#1}}}

\newcommand{\rwtdavecell}[1]{\tensor*[]{\rwt{\updave}}{_{#1}}}

\newcommand{\rwhdavecell}[1]{\tensor*[]{\rwh{\updave}}{_{#1}}}

\newcommand{\exactWICcell}[1]{\tensor*[]{\exactWIC}{_{#1}}}

\newcommand{\nexangWICcell}[1]{\tensor*[]{\nexangWIC}{_{#1}}}

\newcommand{\Exactcatcell}[1]{\tensor*[]{\Exactcat}{_{#1}}}

\newcommand{\ICExactcatcell}[1]{\tensor*[]{\ICExactcat}{_{#1}}}

\newcommand{\WICExactcatcell}[1]{\tensor*[]{\WICExactcat}{_{#1}}}

\newcommand{\Setd}[2]{\tensor*[]{\Set{#1}}{_{#2}}}

\newcommand{\Gammad}[1]{\tensor*[]{\Gamma}{_{#1}}}

\newcommand{\rwtGammad}[1]{\tensor*[]{\rwt{\Gamma}}{_{#1}}}

\newcommand{\rwhGammad}[1]{\tensor*[]{\rwh{\Gamma}}{_{#1}}}

\newcommand{\Deltad}[1]{\tensor*[]{\Delta}{_{#1}}}

\newcommand{\Thetad}[1]{\tensor*[]{\Theta}{_{#1}}}

\newcommand{\Phid}[1]{\tensor*[]{\Phi}{_{#1}}}

\newcommand{\curlbraced}[1]{\tensor*[]{\}}{_{#1}}}

\newcommand{\braced}[1]{\tensor*[]{)}{_{#1}}}

\newcommand{\braceu}[1]{\tensor*[]{)}{^{#1}}}

\newcommand{\vertcomp}{\tensor*[]{\circ}{_{v}}}

\newcommand{\horicomp}{\tensor*[]{\circ}{_{h}}}

\makeatletter

	\renewcommand{\andify}{%
		\nxandlist{\unskip, }{\unskip{} \@@and~}{\unskip \penalty-2 \space \@@and~}}

	    \newenvironment{acknowledgements}{%
	    \renewcommand\abstractname{\textbf{Acknowledgements}}
\global\setbox\abstractbox=\vtop \bgroup
\normalfont\Small
\list{}{\labelwidth\z@
\leftmargin3pc \rightmargin\leftmargin
\listparindent\normalparindent \itemindent\z@
\parsep\z@ \@plus\p@

}%
\item[\hskip\labelsep\scshape\abstractname.]%
}{%
\endlist\egroup
\ifx\@setabstract\relax \@setabstracta \fi
}

 \def\@setaddresses{\par
     \nobreak \begingroup
     \setstretch{1} 
     \setlength{\parindent}{0cm}
     \footnotesize
     \def\author##1{\nobreak\addvspace\bigskipamount}%
     \def\\{\unskip, \ignorespaces}%
     \interlinepenalty\@M
     \def\address##1##2{\begingroup
         \par\addvspace\bigskipamount
     \@ifnotempty{##1}{(\ignorespaces##1\unskip) }%
     {\scshape\ignorespaces##2}\par\endgroup}%
     \def\curraddr##1##2{\begingroup
     \@ifnotempty{##2}{\nobreak\curraddrname
     \@ifnotempty{##1}{, \ignorespaces##1\unskip}\/:\space
     ##2\par}\endgroup}%
     \def\email##1##2{\begingroup
     \@ifnotempty{##2}{\nobreak\emailaddrname
     \@ifnotempty{##1}{, \ignorespaces##1\unskip}\/:\space
     \ttfamily##2\par}\endgroup}%
     \def\urladdr##1##2{\begingroup
     \def~{\char'\~}%
     \@ifnotempty{##2}{\nobreak\urladdrname
     \@ifnotempty{##1}{, \ignorespaces##1\unskip}\/:\space
     \ttfamily##2\par}\endgroup}%
     \addresses
     \endgroup
 }

\let\oldtocsection=\tocsection
\let\oldtocsubsection=\tocsubsection

\renewcommand{\tocsection}[2]{\hspace{0em}\vspace*{1pt}\oldtocsection{#1}{#2}}

    \renewcommand{\tocsubsection}[2]{\hspace{21pt}\vspace*{0pt}\oldtocsubsection{#1}{#2}}

\makeatother

\allowdisplaybreaks 

\begin{document}

\title{The category of extensions and idempotent completion}

    \author[Bennett-Tennenhaus]{Raphael Bennett-Tennenhaus}
        \address{
		Department of Mathematics\\
		Aarhus University\\
		8000 Aarhus C\\
		Denmark}
        \email{raphaelbennetttennenhaus@gmail.com}
        
    \author[Haugland]{Johanne Haugland}
        \address{Department of Mathematical Sciences\\ 
        NTNU\\ 
        NO-7491 Trondheim\\ 
        Norway}
        \email{johanne.haugland@ntnu.no}

    \author[Sand\o y]{Mads Hustad Sand\o y}
        \address{Department of Mathematical Sciences\\ 
        NTNU\\ 
        NO-7491 Trondheim\\ 
        Norway}
        \email{mads.sandoy@ntnu.no}
        
    \author[Shah]{Amit Shah}
        \address{
		Department of Mathematics\\
		Aarhus University\\
		8000 Aarhus C\\
		Denmark}
        \email{amit.shah@math.au.dk}
\date{\today}
\keywords{
Category of extensions, additive category, biadditive functor,
exact category, 
idempotent completion, 
$n$-exangulated category, 
$n$-exangulated functor, 
$n$-exangulated natural transformation,
$2$-category,
$2$-functor}
\subjclass[2020]{18E05, 18E10, 18G80, 18N10, 18G99}

{\setstretch{1}
\begin{abstract}
Building on previous work, we study the splitting of idempotents in the category of extensions $\mathbb{E}\operatorname{-Ext}(\mathcal{C})$ associated to a pair $(\mathcal{C},\mathbb{E})$ of an additive category and a biadditive functor to the category of abelian groups. In particular, we show that idempotents split in \mbox{$\mathbb{E}\operatorname{-Ext}(\mathcal{C})$} whenever they do so in $\mathcal{C}$, allowing us to prove that idempotent completions and extension categories are compatible constructions in a \mbox{$2$-category}-theoretic sense. Furthermore, we show that the exact category obtained by first taking the idempotent completion of an $n$-exangulated category $(\mathcal{C},\mathbb{E},\mathfrak{s})$, in the sense of Klapproth--Msapato--Shah, and then considering its category of extensions is equivalent to the exact category obtained by first passing to the extension category and then taking the idempotent completion. These two different approaches yield a pair of $2$-functors each taking small $n$-exangulated categories to small idempotent complete exact categories. The collection of equivalences that we provide constitutes a $2$-natural transformation between these $2$-functors. Similar results with no smallness assumptions and regarding weak idempotent completions are also proved.
\end{abstract}
}


\maketitle

\setcounter{tocdepth}{1}
\vspace{-0.5cm}
{\setstretch{1}
\tableofcontents
}
\vspace{-1cm}


\section{Introduction}
\label{sec:introduction}
An additive category is called \emph{idempotent complete} given that every idempotent morphism \emph{splits} (see \cref{defn:split-idempotents}), or equivalently if every idempotent admits a kernel (see e.g.\ \cite[Prop.~3.10]{Shah-Krull-Remak-Schmidt}). The study of idempotent complete categories dates back to work by Karoubi \cite{Karoubi-algebres-de-Clifford-et-K-theorie}, in which it was shown that an additive category $\CC$ can be naturally embedded into an idempotent complete category $\rwt{\CC}$, often called its \emph{Karoubi envelope} or its \emph{idempotent completion} (see \cref{defn:karoubi-envelope} and \cref{prop:inclusion-functor-into-Karoubi-envelope-is-ff-exact}).

The splitting of idempotents plays an important role in contemporary algebraic geometry, homological algebra, representation theory and category theory. Indeed, it is intimately connected to the \emph{Krull--Remak--Schmidt property} and \emph{Krull--Schmidt categories} (see \cite[Cor.~A.2]{ChenYeZhang-Algebras-of-derived-dimension-zero}, \cite[Cor.~4.4]{Krause-KS-cats-and-projective-covers}).  Krull--Schmidt categories constitute a particularly nice class of examples of idempotent complete categories. In such a category, every object decomposes, essentially uniquely, into a finite direct sum of indecomposable objects with local endomorphism rings. Splitting of idempotents is often a crucial standing assumption when approaching representation theory of finite-dimensional algebras from a categorical or geometrical perspective (see e.g.\ \cite{Atiyah-KS-theorem-with-apps-to-sheaves,Auslander-Rep-theory-of-Artin-algebras-I,GabrielRoiter-reps-of-finite-dimensional-algebras,Haugland-the-grothendieck-group-of-an-n-exangulated-category,Haugland-auslander-reiten-triangles-and-grothendieck-groups-of-triangulated-categories,Jorgensen-abelian-subcategories-of-triangulated-categories-induced-by-simple-minded-systems,Krause-KS-cats-and-projective-covers}). Furthermore, a generalisation of the Krull--Remak--Schmidt property was given by Azumaya \cite[Thm.~1]{Azumaya-on-generalised-semiprimary-rings}, which has since been used in topological data analysis in the study of persistence homology (see e.g.\ \cite{Botnan-Crawley-Boevey-decomposition-of-persistence-modules}).

Abelian, or more generally exact, and triangulated categories appear in various areas of mathematics, including functional analysis and mathematical physics, and are of fundamental interest in representation theory and related areas (see e.g.\ \cite{Homological-mirror-symmetry-lecture-notes-in-physics,krause,ProsmansSchneiders-a-homological-study-of-bornological-spaces}). Recently, Nakaoka--Palu introduced \emph{\mbox{extriangulated} categories} as a simultaneous generalisation of exact and triangulated \mbox{categories} \cite{NakaokaPalu-extriangulated-categories-hovey-twin-cotorsion-pairs-and-model-structures}, and showed that extension-closed subcategories of triangulated categories, which may fail to be triangulated subcategories, carry an extriangulated structure. Herschend--Liu--Nakaoka \cite{HerschendLiuNakaoka-n-exangulated-categories-I-definitions-and-fundamental-properties} then introduced \emph{\mbox{$n$-exangulated} categories} as a higher-dimensional analogue of extriangulated categories in the context of higher homological algebra. An $n$-exangulated category for an integer $n \geq 1$ is a triplet $(\CC,\BE,\fs)$ consisting of an additive category $\CC$, a biadditive functor \mbox{$\BE\colon \CCu{\op}\times\CC \to \Ab$} (where $\Ab$ denotes the category of abelian groups), and a \emph{realisation} $\fs$ of $\BE$. Note that a category is $1$-exangulated if and only if it is extriangulated \cite[Prop.~4.3]{HerschendLiuNakaoka-n-exangulated-categories-I-definitions-and-fundamental-properties}. Important classes of examples of $n$-exangulated categories for higher $n$ include $n$-exact categories \cite{Jasso-n-abelian-and-n-exact-categories} and $(n+2)$-angulated categories \cite{GeissKellerOppermann-n-angulated-categories}. 

Suppose that $(\CC,\BE,\fs)$ is an $n$-exangulated category. For each pair of objects \mbox{$A,C\inn\CC$}, elements of $\BE(C,A)$ are called \emph{$\BE$-extensions}. The \emph{category of extensions} associated to $(\CC,\BE,\fs)$, denoted by $\catext{\BE}{\CC}$, has all $\BE$-extensions as its objects, and the morphisms are morphisms of $\BE$-extensions; see Subsection~\ref{subsec:cat-of-extensions}. In a previous article, the authors showed that this category can be equipped with a natural exact structure $\CXd{\BE}$, giving rise to an exact category \mbox{$(\catext{\BE}{\CC},\CXd{\BE})$}; see \cite[Prop.~3.2]{Bennett-TennenhausHauglandSandoyShah-the-category-of-extensions-and-a-characterisation-of-n-exangulated-functors}. Moreover, we demonstrated that $\catext{\BE}{\CC}$ encodes important structural information. As an example, this perspective leads to a full characterisation of $n$-exangulated functors between $n$-exangulated categories; see \cite[Thm.~A]{Bennett-TennenhausHauglandSandoyShah-the-category-of-extensions-and-a-characterisation-of-n-exangulated-functors}. In the present paper, we improve the understanding of the relationship between an $n$-exangulated category and its associated category of extensions by studying the splitting of idempotents. Note that our results in \cref{sec:cat-of-extensions-and-splitting-idempotents}, and in particular \cref{prop-A} and \cref{thm-B} below, hold more generally for any pair $(\CC,\BE)$ consisting of an additive category $\CC$ and a biadditive functor \mbox{$\BE\colon \CCu{\op}\times\CC \to \Ab$}.

\cref{prop-A} asserts that the splitting of idempotents in $\catext{\BE}{\CC}$ is inherited from $\CC$. This is our first main result, and it plays an important role in the paper and is a key step in the formulation of \cref{thm-C}.

\begin{propx}[See \cref{prop:idempotents-split-category-of-extensions}] \label{prop-A} 
If $\CC$ is idempotent complete, then $\catext{\BE}{\CC}$ is also idempotent complete.
\end{propx}

As a consequence of \cref{prop-A}, we obtain that given certain finiteness assumptions on $\CC$, the Krull--Remak--Schmidt property for $\CC$ implies the same property for $\catext{\BE}{\CC}$; see \cref{cor-KRS-property}, cf.\ \cite[p.~670]{DraxlerReitenSmaloSolberg-exact-categories-and-vector-space-categories}, \cite[p.~335]{GabrielNazarovaRoiterSergeichukVossieck-tame-and-wild-subspace-problems}.

It is shown in \cite{KlapprothMsapatoShah-Idempotent-completions-of-n-exangulated-categories} that if $(\CC,\BE,\fs)$ is an $n$-exangulated category, then the idempotent completion $\rwt{\CC}$ of $\CC$ admits an $n$-exangulated structure $(\rwt{\CC},\rwt{\BE},\rwt{\fs}\mspace{1mu})$; see \cref{sec:idempotent-completion-n-exangulated-cats}. Our second main result, given as \cref{thm-B} below, demonstrates that idempotent completions and extension categories are compatible constructions. More precisely, the category obtained by first taking the idempotent completion and then considering its category of extensions is equivalent to first passing to the extension category and then taking the idempotent completion. 

\begin{thmx}[See \cref{thm:cat-of-extensions-of-IC-isomorphic-to-IC-of-cat-of-extensions}] \label{thm-B} 
The category $\catext{\rwt{\BE}}{\rwt{\CC}}$ 
is equivalent to the idempotent completion of the category $\catext{\BE}{\CC}$.
\end{thmx}

\cref{prop-A} and \cref{thm-B} are both used in order to obtain the $2$-category-theoretic result \cref{thm-C}, which builds a bridge between the $2$-categorical framework established in \cite{Bennett-TennenhausHauglandSandoyShah-the-category-of-extensions-and-a-characterisation-of-n-exangulated-functors} and the results on the idempotent completion of an $n$-exangulated category from \cite{KlapprothMsapatoShah-Idempotent-completions-of-n-exangulated-categories}. To discuss a $2$-category of $n$-exangulated categories, we use notions of morphisms between $n$-exangulated categories and of morphisms between such morphisms. Structure-preserving functors between $n$-exangulated categories as introduced in \cite{Bennett-TennenhausShah-transport-of-structure-in-higher-homological-algebra} are known as \emph{$n$-exangulated functors}. We viewed the theory of $n$-exangulated categories from a $2$-categorical perspective in \cite{Bennett-TennenhausHauglandSandoyShah-the-category-of-extensions-and-a-characterisation-of-n-exangulated-functors} by defining \emph{$n$-exangulated natural transformations} between $n$-exangulated functors and establishing the $2$-category $\exang{n}$ of small $n$-exangulated categories \cite[Cor.~4.15]{Bennett-TennenhausHauglandSandoyShah-the-category-of-extensions-and-a-characterisation-of-n-exangulated-functors}. Furthermore, we constructed a $2$-functor \mbox{$\updave \colon \exang{n} \to \exactcat$} to the category of small exact categories \cite[Thm.~D]{Bennett-TennenhausHauglandSandoyShah-the-category-of-extensions-and-a-characterisation-of-n-exangulated-functors}, which sends a $0$-cell $(\CC,\BE,\fs)$ in $\exang{n}$ to the $0$-cell \mbox{$(\catext{\BE}{\CC},\CXd{\BE})$} in $\exactcat$; see \cref{def:updave}. A consequence of \cref{prop-A} is that $\updave$ restricts to a $2$-functor \mbox{$\rwt{\updave}\colon \ICexang{n} \to \ICexactcat$} from the $2$-category of small idempotent complete $n$-exangulated categories to the $2$-category of small idempotent complete exact categories. The last observation needed in order to state \cref{thm-C} is that taking idempotent completions yields $2$-functors \mbox{$\exactcomplete\colon \exactcat \to \ICexactcat$} and \mbox{$\nexangcomplete\colon\exang{n}\to\ICexang{n}$}; see \cref{thm:2-functor-of-exact-categories} and \cref{thm:2-functor-of-n-exangulated-categories}, respectively. 

\begin{thmx}[See \cref{cor:commutativity-of-2-square}] \label{thm-C}
Consider the diagram  
\[
\begin{tikzcd}
\exang{n} 
    \arrow{r}{\updave}
    \arrow{d}[swap]{\nexangcomplete} 
& \exactcat 
    \arrow{d}{\exactcomplete}
    \\
\ICexang{n} 
    \arrow{r}{\rwt{\updave}}
& \ICexactcat
\end{tikzcd}
\]
of $2$-categories and $2$-functors. There is a $2$-natural transformation \mbox{$\rwt{\updave}\nexangcomplete \Rightarrow \exactcomplete\updave$} consisting of exact equivalences.
\end{thmx}

Similar results as above hold also for weak idempotent completions; see \cref{sec:weak-idempotent-completion}. We remark that even though \cref{prop:idempotents-split-category-of-extensions-WIC} is an expected analogue of \cref{prop-A} in this setup, the method of proof is different and relies on a previous result of the authors from \cite{Bennett-TennenhausHauglandSandoyShah-the-category-of-extensions-and-a-characterisation-of-n-exangulated-functors}.

\begin{rem}
\label{rem:set-theory-remark}
We note that \cref{thm-C} follows from a more general result, namely \cref{thm:commutative-square}, in which no smallness assumption is required. In this article the term `category' does not require the collections of morphisms to form sets. In other words, the categories we consider need not be \emph{locally small}. Just as explained in \cite[Rem.~4.13]{Bennett-TennenhausHauglandSandoyShah-the-category-of-extensions-and-a-characterisation-of-n-exangulated-functors}, the reason for the restriction to small categories in \cref{thm-C} (and indeed in this introduction entirely) is to be able to use the terminology of $2$-categories and $2$-functors in a way that is consistent with the existing literature. 
\end{rem}


\subsection*{Structure of the paper}
In \cref{sec:idempotent-completion-exact-cats} we recall the construction of the idempotent completion of an exact category and use this to establish the $2$-functor $\exactcomplete$ from \cref{thm-C}. Analogously, the $2$-functor $\nexangcomplete$ is defined in \cref{sec:idempotent-completion-n-exangulated-cats} using the idempotent completion of an $n$-exangulated category in the sense of \cite{KlapprothMsapatoShah-Idempotent-completions-of-n-exangulated-categories}. In \cref{sec:cat-of-extensions-and-splitting-idempotents} we recall how to form the category of extensions, and prove \cref{prop-A} and \cref{thm-B}. In \cref{sec:2-categorical-perspective} we present the definition of the $2$-functor $\updave$ from \cite{Bennett-TennenhausHauglandSandoyShah-the-category-of-extensions-and-a-characterisation-of-n-exangulated-functors} and show how the main results of the previous sections culminate in \cref{thm-C}. \cref{sec:weak-idempotent-completion} concerns the weak idempotent completion. 

\subsection*{Conventions and notation}

Throughout this paper, let $n\geq 1$ denote a positive integer. Given objects $X$ and $Y$ in a category $\CC$, we write $\CC(X,Y)$ for the collection of morphisms from $X$ to $Y$ in $\CC$. Functors are always assumed to be covariant. We let $\Ab$ denote the category of abelian groups. 


\section{Idempotent completion of exact categories yields a \texorpdfstring{$2$}{2}-functor}
\label{sec:idempotent-completion-exact-cats}

The aim for this section is to explicitly relate the construction of the idempotent completion of an exact category to a $2$-categorical framework, establishing the $2$-functor $\exactcomplete$ which is part of \cref{thm-C} in \cref{sec:introduction}. We start by following B{\"{u}}hler \cite[Sec.~6]{Buhler-exact-categories} in recalling the idempotent completion (or Karoubi envelope). We also refer to Borceux \cite{Borceux-handbook-1}. 

Throughout the section, let $\CC$ denote an additive category. An \emph{idempotent} in $\CC$ is a morphism \mbox{$e\colon X\to X$} for some object $X\inn\CC$ satisfying \mbox{$e^2 = e$}. Splitting of idempotents, as defined below, plays a central role in this article.

\begin{defn} 
\label{defn:split-idempotents} 
(See \cite[Defs.~6.5.1, 6.5.3]{Borceux-handbook-1}.) An idempotent \mbox{$e \colon X \to X$} in $\CC$ \emph{splits} if there exist morphisms \mbox{$r \colon X \to Y$} and \mbox{$s \colon Y \to X$} such that $sr=e$ and $rs=\id{Y}$. The category $\CC$ is \emph{idempotent complete}, or has \emph{split idempotents}, if each idempotent in $\CC$ splits. 
\end{defn}

Even though the additive category $\CC$ need not have split idempotents, it can always be embedded into an idempotent complete category. This is due to Karoubi \cite[Sec.~1.2]{Karoubi-algebres-de-Clifford-et-K-theorie}. 

\begin{defn}
\label{defn:karoubi-envelope}
(See \cite[Rem.~6.3, Def.~6.4]{Buhler-exact-categories}.) 
Define a category $\rwt{\CC}$ as follows. The objects of $\rwt{\CC}$ are pairs $(X,e)$ for each object $X\inn\CC$ and each idempotent \mbox{$e\inn\Endd{\CC}(X)$}. Given objects $(X,\ed{X})$ and $(Y,\ed{Y})$ in $\rwt{\CC}$, the collection \mbox{$\rwt{\CC}((X,\ed{X}),(Y,\ed{Y}))$} of morphisms from $(X,\ed{X})$ to $(Y,\ed{Y})$ consists of triplets $(\ed{Y},f,\ed{X})$ such that $f \inn \CC(X,Y)$ satisfies \mbox{$f \ed{X} = f = \ed{Y} f$}. The composition of \mbox{$(\ed{Y},f,\ed{X})\inn\rwt{\CC}((X,\ed{X}), (Y,\ed{Y}))$} and \mbox{$(\ed{Z},g,\ed{Y})\inn\rwt{\CC}((Y,\ed{Y}), (Z,\ed{Z}))$} is given by 
\[
(\ed{Z},g,\ed{Y})\circ(\ed{Y},f,\ed{X}) \deff (\ed{Z}, gf, \ed{X}).
\]
It is clear that this composition is associative. The identity $\id{(X,e)}$ of $(X,e)\inn \rwt{\CC}$ is the morphism $(e,e,e)$. The category $\rwt{\CC}$ is called the \emph{idempotent completion of $\CC$}. 
\end{defn}

The category $\rwt{\CC}$ is additive with biproduct given by \mbox{$(X, \ed{X}) \oplus (Y, \ed{Y}) = (X \oplus Y, \ed{X} \oplus \ed{Y})$}. It is also idempotent complete; see \cite[Rem.~6.3]{Buhler-exact-categories} for details. There is a canonical additive inclusion functor \mbox{$\SId{\CC} \colon \CC \to \rwt{\CC}$} defined by setting \mbox{$\SId{\CC}(X) \deff (X, \id{X})$} for $X \inn \CC$ and \mbox{$\SId{\CC}(f) \deff (\id{Y}, f,\id{X})$} for $f \inn \CC(X,Y)$. This functor is $2$-universal among additive functors from $\CC$ to idempotent complete categories; see \cite[Prop.~6.10]{Buhler-exact-categories}.

Let $\CCu{\rightarrow\rightarrow}$ denote the category of composable morphisms in $\CC$, and note that a functor \mbox{$\CC\to\CD$} induces a functor \mbox{$\CCu{\rightarrow\rightarrow}\to\CDu{\rightarrow\rightarrow}$}. Now suppose $(\CC,\CX)$ is an exact category. In particular, the exact structure $\CX$ is a collection of objects in $\CCu{\rightarrow\rightarrow}$. One can define an exact structure $\rwt{\CX}$ on $\rwt{\CC}$ by declaring an object in $\rwtCCu{\rightarrow\rightarrow}$ to be in $\rwt{\CX}$ if it is a direct summand of an object belonging to the image of $\CX$ under the functor \mbox{$\CCu{\rightarrow\rightarrow}\to \rwtCCu{\rightarrow\rightarrow}$} induced by \mbox{$\SId{\CC} \colon \CC \to \rwt{\CC}$}.

\begin{prop}
\label{prop:inclusion-functor-into-Karoubi-envelope-is-ff-exact}
(See \cite[Rem.~6.3, Prop.~6.13]{Buhler-exact-categories}.) 
The pair $(\rwt{\CC},\rwt{\CX})$ forms an exact category, and \mbox{$\SId{\CC} \colon (\CC,\CX) \to (\rwt{\CC},\rwt{\CX})$} is a fully faithful exact functor that reflects exactness. 
\end{prop}

Let $\SF\colon \CC \to \CD$ be an additive functor. Following \cite[Rem.~6.6]{Buhler-exact-categories}, there is an induced additive functor \mbox{$\rwt{\SF}\colon \rwt{\CC} \to \rwt{\CD}$} given by 
\begin{equation}
\label{eqn:IC-of-functor}
\rwt{\SF}(X,e) \deff (\SF X, \SF e)
\hspace{1cm}
\text{and}
\hspace{1cm}
\rwt{\SF}(\ed{Y},f,\ed{X}) \deff (\SF \ed{Y}, \SF f, \SF \ed{X}).
\end{equation}
We refer to $\rwt{\SF}$ as the \emph{completion} of $\SF$. If $\SF \colon (\CC,\CX) \to (\CD,\CY)$ is an exact functor, then \mbox{$\rwt{\SF}\colon (\rwt{\CC},\rwt{\CX}) \to (\rwt{\CD},\rwt{\CY})$} is also exact; see the proof of \cite[Prop.~6.13]{Buhler-exact-categories}. 

In order to view the constructions above in a $2$-categorical framework, we recall some terminology. A \emph{$2$-category} is a collection of \emph{$0$-cells}, \emph{$1$-cells} and \emph{$2$-cells} satisfying certain axioms; see e.g.\ \cite[p.~273]{MacLane-categories-for-the-working-mathematician} or \cite[Sec.~2.3]{JohnsonYau-2-dimensional-categories}. One should think of $0$-cells, $1$-cells and $2$-cells as objects, morphisms between objects and morphisms between morphisms, respectively. A \mbox{$2$-category} has two notions of composition of $2$-cells: \emph{vertical} and \emph{horizontal}. Using the setup below, we recall these notions in the case of natural transformations. We use the Hebrew letters $\Beth$ (beth) and $\Daleth$ (daleth) for natural transformations of additive functors.

\begin{setup}
\label{setup:section4} 
For the rest of this section, we consider 
additive categories 
    $\CC,\CD,\CE$, 
additive functors 
    \mbox{$\SF,\SG,\SH\colon\CC\to\CD$} 
    and 
    \mbox{$\SL,\SM\colon\CD\to\CE$}, 
and 
natural transformations 
    \mbox{$\Beth\colon\SF\Rightarrow\SG$}, 
    \mbox{$\Beth'\colon\SG\Rightarrow\SH$}
    and 
    \mbox{$\Daleth\colon\SL\Rightarrow\SM$} 
as indicated in the diagram 
\[
\begin{tikzcd}
&{}
    \arrow[Rightarrow,
            shorten <= 6pt,
            shorten >= 6pt,
            yshift = -2pt
            ]{d}
                [xshift=2pt, yshift=-1pt]{\Beth}
&
&{}
    \arrow[Rightarrow,
            shorten <= 10pt, 
            shorten >= 12pt,
            xshift=0pt,
            yshift = -2pt
            ]{dd}
                [xshift=1pt, yshift=1pt]{\Daleth}
&
\\
\CC
    \arrow[bend left=50]{rr}
            [description, xshift=0pt, yshift=0pt]{\SF}
    \arrow{rr}
            [description]{\SG}
    \arrow[bend right=50]{rr}
            [description, xshift=0pt, yshift=0pt]{\SH}
&{}
    \arrow[Rightarrow,
            shorten <= 3pt,
            shorten >= 6pt,
            yshift=4pt
            ]{d}
                [xshift=2pt, yshift=3pt]{\Beth'}
&\CD
    \arrow[bend left=50]{rr}
            [description, xshift=0pt, yshift=0pt]{\SL}
    \arrow[bend right=50]{rr}
            [description, xshift=0pt, yshift=0pt]{\SM}
&{}
&\CE.
\\
&{}&&{}&
\end{tikzcd}
\]
\end{setup}

\begin{defn}
\label{notn:vertical-composition-and-horizontal-composition-of-2-cells} 
(See \cite[pp.~40, 42]{MacLane-categories-for-the-working-mathematician}.) 
The \emph{vertical composition} of $\Beth$ and $\Beth'$  is the natural transformation \mbox{$\Beth'\vertcomp\Beth\colon \SF\Rightarrow \SH$} given by \mbox{$\tensor*[]{(\Beth'\vertcomp\Beth)}{_{X}}\deff\tensor*[]{\Beth}{_{X}^{\prime}}\Bethd{X}$} for each $X\inn\CC$. The \emph{horizontal composition} of $\Beth$ and $\Daleth$ is the natural transformation \mbox{$\Daleth\horicomp\Beth\colon \SL\SF\Rightarrow \SM\SG$} defined by \mbox{$(\Daleth\horicomp\Beth\braced{X}\deff \Dalethd{\SG X}\circ(\SL \Bethd{X})$} for each $X\inn\CC$. 
\end{defn}

As described in \cite[Rem.~6.7]{Buhler-exact-categories}, the natural transformation $\Beth \colon \SF \Rightarrow \SG$ induces a natural transformation $\rwt{\Beth}\colon \rwt{\SF}\Rightarrow\rwt{\SG}$ as follows. Given $(X,e) \inn \rwt{\CC}$, there are the morphisms \mbox{$(\id{X},e,e)\colon (X,e) \to (X,\id{X})$} and \mbox{$(e,e,\id{X})\colon (X,\id{X}) \to (X,e)$}. Put 
\[
\rwtBethd{(X,e)} \deff 
	\rwt{\SG}(e,e,\id{X})\circ (\id{\SG X}, \Bethd{X}, \id{\SF X}) \circ \rwt{\SF}(\id{X},e,e)
	= (\SG e,(\SG e)\Bethd{X}\SF e,\SF e)
\]
as indicated in the diagram
\[
\begin{tikzcd}[column sep=3cm]
\rwt{\SF}(X,e) 
	\arrow[dotted]{r}{\rwtBethd{(X,e)}}
	\arrow{d}[swap]{\rwt{\SF}(\id{X},e,e)}
& \rwt{\SG}(X,e) \\
(\SF X,\id{\SF X})
	\arrow{r}[swap]{(\id{\SG X}, \Bethd{X}, \id{\SF X})}
& (\SG X,\id{\SG X}).
	\arrow{u}[swap]{\rwt{\SG}(e,e,\id{X})}
\end{tikzcd}
\]
It is straightforward to check that $\rwt{\Beth}$ is natural. We refer to $\rwt{\Beth}$ as the \emph{completion} of $\Beth$.

\begin{notn}
\label{not:Exactcat-definition}
We write $\Exactcat$ for the collection of $0$-cells, $1$-cells and $2$-cells consisting of exact categories, exact functors and natural transformations, respectively. For $i\inn\{0,1,2\}$, we denote the collection of $i$-cells by $\Exactcatcell{i}$. Given $0$-cells $(\CC,\CX)$ and $(\CD,\CY)$, there is a category $\Exactcat((\CC,\CX), (\CD,\CY))$ with $1$-cells of the form \mbox{$(\CC,\CX) \to (\CD,\CY)$} as objects, and where morphisms and composition are given by $2$-cells and vertical composition. We note that for an object $\SF$ in $\Exactcat((\CC,\CX), (\CD,\CY))$, its identity morphism is the identity natural transformation $\id{\SF}\colon\SF\Rightarrow\SF$ given by 
$\tensor*[]{\Set{(\id{\SF}\tensor*[]{)}{_{X}}\deff\id{\SF X}}}{_{X\inn\CC}}$. There is a \mbox{$2$-category} $\exactcat$ determined by the $0$-cells in $\Exactcat$ which are small categories. We furthermore write $\ICExactcat$ and $\ICexactcat$ when restricting to idempotent complete $0$-cells in $\Exactcat$ and $\exactcat$, respectively, and note that also $\ICexactcat$ is a $2$-category. 
\end{notn}

A \emph{$2$-functor} between two $2$-categories is an assignment of $i$-cells in the domain category to $i$-cells in the codomain category for $i\inn\{1,2,3\}$, satisfying some compatibility conditions; see e.g.\ \cite[p.~278]{MacLane-categories-for-the-working-mathematician} or \cite[Prop.~4.1.8]{JohnsonYau-2-dimensional-categories}. We now begin to construct the $2$-functor $\exactcomplete$ used in \cref{thm-C} in \cref{sec:introduction}. 

\begin{defn}
\label{def:exactcomplete}
Let $\exactcomplete = (\exactcompletecell{0},\exactcompletecell{1},\exactcompletecell{2}) \colon \Exactcat \to \ICExactcat$ be defined by the assignments \mbox{$\exactcompletecell{i}\colon \Exactcatcell{i} \to \ICExactcatcell{i}$}, where:
\[
\exactcompletecell{0}  (\CC,\CX)  
    \deff (\rwt{\CC},\rwt{\CX}),
\hspace{1cm}    
\exactcompletecell{1}  (\SF)  
    \deff \rwt{\SF},
\hspace{1cm}
\exactcompletecell{2} ( \Beth ) 
    \deff \rwt{\Beth}.
\]
\end{defn}

If one ignores the set-theoretic issue described in \cref{rem:set-theory-remark}, then the theorem below should be interpreted as showing that $\exactcomplete \colon \Exactcat \to \ICExactcat$ is a $2$-functor.

\begin{thm}
\label{thm:2-functor-of-exact-categories}
The following statements hold for the assignments
$\exactcompletecell{0}$, $\exactcompletecell{1}$ and $\exactcompletecell{2}$.
\begin{enumerate}[label=\textup{(\roman*)}] 
    \item\label{item:exactcomplete-0-and-1-give-functor-on-underyling-1-categories}
    The pair  
    $(\exactcompletecell{0}, \exactcompletecell{1})$ 
    defines a functor 
    $\Exactcat\to\ICExactcat$.
    \item\label{item:exactcomplete-1-and-2-give-functor} 
    The pair 
    $(\exactcompletecell{1}, \exactcompletecell{2})$ 
    defines a functor 
	\mbox{$\Exactcat((\CC,\CX), (\CD,\CY)) \to \ICExactcat((\rwt{\CC},\rwt{\CX}), (\rwt{\CD},\rwt{\CY}))$}
    whenever $(\CC,\CX)$ and $(\CD,\CY)$ are exact categories.
    \item\label{item:exactcomplete2-preserves-horizontal-composition} 
    The assignment $\exactcompletecell{2}$ preserves horizontal composition.
\end{enumerate}
In particular, restricting $\exactcomplete$ to small categories yields a $2$-functor \mbox{$\exactcat \to \ICexactcat$}. 
\end{thm}

\begin{proof}
It follows from the discussions above that the assignments are well-defined. 
Checking functoriality in \ref{item:exactcomplete-0-and-1-give-functor-on-underyling-1-categories} is straightforward. 
The assignment $\exactcompletecell{2}$ is compatible with vertical and horizontal composition by \cite[Rem.~6.8]{Buhler-exact-categories}, and checking $\rwt{\id{\SF}} = \id{\rwt{\SF}}$ is straightforward, so \ref{item:exactcomplete-1-and-2-give-functor} and \ref{item:exactcomplete2-preserves-horizontal-composition} hold.
\end{proof}


\section{Idempotent completion of \texorpdfstring{$n$}{n}-exangulated categories yields a \texorpdfstring{$2$}{2}-functor}
\label{sec:idempotent-completion-n-exangulated-cats}

In this section we describe how taking the idempotent completion of an $n$-exangulated category in the sense of \cite{KlapprothMsapatoShah-Idempotent-completions-of-n-exangulated-categories} relates to a $2$-categorical framework. This is done by constructing the $2$-functor $\nexangcomplete$ from \cref{thm-C} in \cref{sec:introduction}. We start by giving an overview of relevant notions and constructions.

Given an additive category $\CC$ and a biadditive functor $\BE \colon \CCu{\op}\times\CC \to \Ab$, an element $\alpha \inn \BE(C,A)$ is called an \emph{$\BE$-extension}. A \emph{morphism of $\BE$-extensions} from $\alpha\inn \BE(C,A)$ to \mbox{$\beta\inn \BE(D,B)$} is a pair $(a,c)$ of morphisms \mbox{$a\colon A\to B$} and \mbox{$c\colon C\to D$} in $\CC$ such that 
\[
\BE(C,a)(\alpha)=\BE(c,A)(\beta).
\]
 
Recall from \cite[Sec.~2]{HerschendLiuNakaoka-n-exangulated-categories-I-definitions-and-fundamental-properties} that an \emph{$n$-exangulated category} $(\CC,\BE,\fs)$ consists of 
\begin{enumerate}[label=\textup{(\roman*)}]
	\item an additive category $\CC$, 
	\item a biadditive functor $\BE \colon \CCu{\op}\times\CC \to \Ab$, and 
	\item an exact realisation $\fs$ of $\BE$ in the sense of \cite[Def.~2.22]{HerschendLiuNakaoka-n-exangulated-categories-I-definitions-and-fundamental-properties}, 
\end{enumerate}
such that axioms (EA$1$), (EA$2$) and (EA$\tensor*[]{2}{^{\op}}$) stated in \cite[Def.~2.32]{HerschendLiuNakaoka-n-exangulated-categories-I-definitions-and-fundamental-properties} are satisfied. 
 
The realisation $\fs$ associates to each $\BE$-extension $\alpha\inn\BE(C,A)$ a certain homotopy class 
\[
\fs(\alpha)
    = [\Xcom]
    = [\begin{tikzcd}
        \Xd{0}
            \arrow{r}
        & \Xd{1}
            \arrow{r}
        & \cdots
            \arrow{r}
        & \Xd{n+1}
    \end{tikzcd}]
\]
of an $(n+2)$-term complex $\Xcom$ in $\CC$ with $\Xd{0} = A$ and $\Xd{n+1} = C$. The pair $\lan \Xcom, \alpha\ran$ is then called a \emph{(distinguished) $n$-exangle}. 

A \emph{morphism $\lan \Xcom, \alpha\ran \to \lan \Yd{\bullet}, \beta \ran$ of $n$-exangles} is given by a morphism \mbox{$(\fd{0},\ldots, \fd{n+1})\colon \Xcom\to \Yd{\bullet}$} of complexes such that $(\fd{0},\fd{n+1}) \colon \alpha \to \beta$ is a morphism of $\BE$-extensions. In this case, the tuple \mbox{$(\fd{0},\ldots, \fd{n+1})$} is said to be a \emph{lift} of $(\fd{0},\fd{n+1})$. 

Suppose throughout this section that $(\CC,\BE,\fs)$ and $(\CD,\BF,\ft)$ are $n$-exangulated categories. An additive functor $\SF\colon\CC\to\CD$ induces a functor $\SFd{\com} \colon \comd{\CC} \to \comd{\CD}$ between the associated categories of complexes. One can define a new biadditive functor \mbox{$\BF(\SFu{\op}-,\SF-)\colon \CCu{\op}\times\CC \to \Ab$}, which we will denote by $\BF(\SF-,\SF-)$. 

\begin{defn}
\label{def:n-exangulated-functor}(See \cite[Def.~2.32]{Bennett-TennenhausShah-transport-of-structure-in-higher-homological-algebra}.)
Let $\SF\colon \CC\to\CD$ be an additive functor and suppose there is a natural transformation
\[
\Gamma = 
    \{\Gammad{(C,A)} \curlbraced{(C,A)\inn\CCu{\scalebox{0.7}{\op}}\times\CC}
    \colon \BE(-,-) \Longrightarrow \BF(\SF-,\SF-).
\]
We call the pair $(\SF,\Gamma) \colon (\CC,\BE,\fs) \to (\CD,\BF,\ft)$ an \emph{$n$-exangulated functor} if, for all $A,C\inn\CC$ and each $\alpha\inn\BE(C,A)$, we have that $\fs(\alpha)=[\Xcom]$ implies $\ft(\Gammad{(C,A)}(\alpha))=[\SFd{\com} \Xcom]$.
\end{defn}

It was demonstrated in \cite[Def.~3.18, Lem.~3.19]{Bennett-TennenhausHauglandSandoyShah-the-category-of-extensions-and-a-characterisation-of-n-exangulated-functors} that one can compose $n$-exangulated functors as follows. Suppose $(\SF, \Gamma)\colon (\CC,\BE,\fs) \to (\CD,\BF,\ft)$ and \mbox{$(\SL, \Phi)\colon (\CD,\BF,\ft) \to (\CE,\BG,\fu)$} are \mbox{$n$-exangulated} functors between $n$-exangulated categories. The \emph{composite} is the $n$-exangulat\-ed functor \mbox{$(\SL,\Phi) \circ (\SF,\Gamma)\deff(\SL\SF, \Phid{\SF\times\SF} \vertcomp \Gamma)$}, where $\Phid{\SF\times\SF}$ is the natural transformation 
\[
\Phid{\SF\times\SF}
    = 
    \{\Phid{(\SF C,\SF A)}\curlbraced{(C,A)\inn\CCu{\scalebox{0.7}{\op}}\times\CC}
        \colon 
    \BF(\SF-,\SF-) 
        \Longrightarrow
    \BG(\SL\SF-,\SL\SF-).
\] 

For $\Gamma$ as above and for $\alpha\inn\BE(C,A)$, we will usually write $\Gamma(\alpha)$ instead of $\Gammad{(C,A)}(\alpha)$. Furthermore, we use the simplified notation $\ad{\BE}\alpha$ (resp.\ $\du{\BE}\alpha$) for the $\BE$-extension \mbox{$\BE(C,a)(\alpha)\inn\BE(C,B)$} (resp.\ \mbox{$\BE(d,A)(\alpha)\inn\BE(D,A)$}) for morphisms $a\colon A\to B$ and $d\colon D\to C$ in $\CC$.

As proved in \cite{KlapprothMsapatoShah-Idempotent-completions-of-n-exangulated-categories}, the idempotent completion $\rwt{\CC}$ of an $n$-exangulated category $(\CC,\BE,\fs)$ admits a canonical $n$-exangulated structure. We use the notation $(\rwt{\CC},\rwt{\BE},\rwt{\fs}\mspace{1mu})$ for the $n$-exangulated category obtained from this construction, and recall the definition of the biadditive functor \mbox{$\rwt{\BE}\colon \rwtCCu{\op} \times \rwt{\CC} \to \Ab$} and the realisation $\rwt{\fs}$ of $\rwt{\BE}$ below. In the case $n=1$, the construction was given by Msapato \cite{Msapato-the-karoubi-envelope-and-weak-idempotent-completion-of-an-extriangulated-category}. 

\begin{defn}
\label{def:bifunctor-on-IC-of-exangulated-category} 
(See \cite[Def.~4.4]{KlapprothMsapatoShah-Idempotent-completions-of-n-exangulated-categories}.) 
For objects $(A,\ed{A}),(C,\ed{C})\inn\rwt{\CC}$, we let
\begin{align*}
\rwt{\BE}((C,\ed{C}),(A,\ed{A}))
    & \deff 
    \Set{ (\ed{A}, \alpha, \ed{C}) 
        | 
    \text{$\alpha\inn\BE(C,A)$ 
    and 
    $(\ed{A}\braced{\BE}   \alpha 
        = \alpha 
        = (\ed{C}\braceu{\BE}   \alpha$}}.
\end{align*}
For morphisms $(\ed{B}, a, \ed{A})\colon (A,\ed{A})\to (B,\ed{B})$ and $(\ed{C},d,\ed{D})\colon (D,\ed{D})\to (C,\ed{C})$ in $\rwt{\CC}$ we put 
\begin{align*}
\rwt{\BE}((\ed{C},d,\ed{D}) , (\ed{B}, a, \ed{A})) 
    \colon 
    \rwt{\BE}((C,\ed{C}), (A,\ed{A})) 
    &\longrightarrow 
    \rwt{\BE}((D, \ed{D}), (B, \ed{B})) \\ 
(\ed{A}, \alpha, \ed{C}) 
    &\longmapsto (\ed{B}, \BE(d,a)(\alpha), \ed{D}).
\end{align*}
\end{defn}

The set $\rwt{\BE}((C,\ed{C}),(A,\ed{A}))$ has an abelian group structure given by 
\[
(\ed{A}, \alpha, \ed{C}) + (\ed{A}, \alpha', \ed{C})
    \deff (\ed{A}, \alpha + \alpha', \ed{C}), 
\]
and \cref{def:bifunctor-on-IC-of-exangulated-category} indeed gives a biadditive functor $\rwt{\BE}\colon \rwtCCu{\op} \times \rwt{\CC} \to \Ab$; see \cite[Rem.~4.5]{KlapprothMsapatoShah-Idempotent-completions-of-n-exangulated-categories}. 

Given a complex $\Xcom$ in $\CC$ and an idempotent morphism $\ed{\bullet} \colon \Xcom \to \Xcom$ of complexes, we follow \cite[Def.~4.15]{KlapprothMsapatoShah-Idempotent-completions-of-n-exangulated-categories} and the discussion immediately thereafter in using the notation $(\Xcom, \ed{\bullet})$ to denote the complex
\begin{equation*}
\label{eqn:Xe-bullet-definition}
\hspace{-4pt}
\begin{tikzcd}[column sep=1.3cm, scale cd=0.9]
(X_0,\ed{0})
    \arrow{rr}{(\ed{1},\ed{1}\dd{0},\ed{0})}
&& (\Xd{1},\ed{1})
    \arrow{rr}{(\ed{2},\ed{2}\dd{1},\ed{1})}
&& \cdots
    \arrow{r}
& (\Xd{n},\ed{n})
    \arrow{rr}{(\ed{n+1},\ed{n+1}\dd{n},\ed{n})}
&& (\Xd{n+1},\ed{n+1})
\end{tikzcd}
\end{equation*}
in $\rwt{\CC}$, where the maps $d_{i}\colon X_{i}\to X_{i+1}$ are the differentials of the complex $\Xcom$. The realisation $\rwt{\fs}$ of $\rwt{\BE}$ is then defined as follows. 

\begin{defn}
\label{def:realisation-of-wt-BE}
(See \cite[Def.~4.20]{KlapprothMsapatoShah-Idempotent-completions-of-n-exangulated-categories}.) 
Let $(\ed{A}, \alpha, \ed{C}) \inn \rwt{\BE}((C,\ed{C}),(A,\ed{A}))$ be arbitrary. Since $\alpha\inn\BE(C,A)$, one may choose $\Xcom$ so that 
$
\fs(\alpha) 
    = [\Xcom]
    = [\begin{tikzcd}[column sep=0.6cm]
        A \arrow{r}
        & \Xd{1} \arrow{r}
        & \cdots \arrow{r}
        & \Xd{n} \arrow{r}
        & C
    \end{tikzcd}].
$
One may also lift $(\ed{A}, \ed{C}) \colon \alpha \to \alpha$ to an idempotent endomorphism $\ed{\bullet}$ of the $n$-exangle $\langle \Xcom, \alpha \rangle$ by \cite[Cor.~4.13]{KlapprothMsapatoShah-Idempotent-completions-of-n-exangulated-categories}. Using this, we define $\rwt{\fs}$ by setting $\rwt{\fs}(\ed{A}, \alpha, \ed{C}) \deff [(\Xcom, \ed{\bullet})]$.
\end{defn}

To see that the assignment $\rwt{\fs}$ from \cref{def:realisation-of-wt-BE} does not rely on the choices involved, see \cite[Rem.~4.21]{KlapprothMsapatoShah-Idempotent-completions-of-n-exangulated-categories}. By \cite[Thm.~A]{KlapprothMsapatoShah-Idempotent-completions-of-n-exangulated-categories}, the triplet $(\rwt{\CC},\rwt{\BE},\rwt{\fs}\mspace{1mu})$ is an $n$-exangulated category and the inclusion \mbox{$(\SId{\CC}, \Gammad{\CC})\colon (\CC,\BE,\fs)\to(\rwt{\CC},\rwt{\BE},\rwt{\fs}\mspace{1mu})$} is an $n$-exangulated functor, where the natural transformation \mbox{$\Gammad{\CC}\colon \BE(-,-)\Rightarrow \rwt{\BE}(\SId{\CC}-,\SId{\CC}-)$} is given by $\alpha \mapsto (\id{A}, \alpha, \id{C})$ for $\alpha\inn \BE(C,A)$. 

Now suppose $(\SF,\Gamma)\colon (\CC,\BE,\fs)\to (\CD,\BF,\ft)$ is an $n$-exangulated functor. Recall that there is an induced additive functor $\rwt{\SF}\colon \rwt{\CC}\to\rwt{\CD}$ as defined in \eqref{eqn:IC-of-functor}. Our next aim is to show that one obtains an $n$-exangulated functor 
\mbox{$(\rwt{\SF},\rwt{\Gamma}) \colon 
(\rwt{\CC},\rwt{\BE},\rwt{\fs}\mspace{1mu})
    \to 
(\rwt{\CD},\rwt{\BF},\rwt{\ft}\mspace{1mu})$} between the idempotent completions. We first need to define a natural transformation \mbox{$\rwt{\Gamma} \colon \rwt{\BE}(-,-)\Rightarrow\rwt{\BF}(\rwt{\SF}-,\rwt{\SF}-)$}.

\begin{defn}
\label{def:completion-of-Gamma}
Set
$
\rwt{\Gamma}\deff \Setd{\rwtGammad{((C,\ed{C}),(A,\ed{A}))}}{((C,\ed{C}),(A,\ed{A}))\inn\rwtCCu{\op}\times\rwt{\CC}},
$
where
\begingroup
\allowdisplaybreaks[0] 
\begin{align*}
\rwtGammad{((C,\ed{C}),(A,\ed{A}))} 
    \colon 
    \rwt{\BE}((C,\ed{C}), (A,\ed{A})) 
    &\longrightarrow 
    \rwt{\BF}(\rwt{\SF}(C, \ed{C}), \rwt{\SF}(A, \ed{A})) \\ 
(\ed{A}, \alpha, \ed{C}) 
    &\longmapsto (\SF \ed{A}, \Gamma(\alpha), \SF \ed{C}).
\end{align*}
\endgroup
\end{defn}

Note that $(\SF \ed{A},\Gamma(\alpha),\SF \ed{C})$ indeed lies in $\rwt{\BF}(\rwt{\SF}(C,\ed{C}),\rwt{\SF}(A,\ed{A}))$, because naturality of $\Gamma$ yields 
\mbox{$
(\SF \ed{A}\braced{\BF}   \Gamma(\alpha) 
	= \Gamma((\ed{A}\braced{\BE}  \alpha)
	= \Gamma(\alpha)
$}
and
$(\SF \ed{C}\braceu{\BF}   \Gamma(\alpha) 
= \Gamma((\ed{C}\braceu{\BE}  \alpha)
	= \Gamma (\alpha)$. 

\begin{warning}
Since $\Gamma\colon\BE(-,-)\Rightarrow \BF(\SF,-,\SF-)$ is a natural transformation, one might wonder if the definition of $\rwt{\Gamma}$ above agrees with the description of the completion of a natural transformation of additive functors from \cref{sec:idempotent-completion-exact-cats}. However, the \emph{biadditive} functors $\BE$ and $\BF(\SF-,\SF-)$ are not necessarily \emph{additive} functors $\CCu{\op}\times\CC\to\Ab$, so we cannot form the completions of them as in \cref{sec:idempotent-completion-exact-cats}. Thus, when we use notation of the form $\rwt{\Gamma}$ for a natural transformation of biadditive functors, it always refers to the construction from \cref{def:completion-of-Gamma}.
\end{warning}

\begin{lem}
\label{prop:completion-of-n-ex-functor}
The pair $(\rwt{\SF},\rwt{\Gamma})$ is an $n$-exangulated functor $(\rwt{\CC},\rwt{\BE},\rwt{\fs}\mspace{1mu})\to (\rwt{\CD},\rwt{\BF},\rwt{\ft}\mspace{1mu})$.
\end{lem}

\begin{proof}
Given a pair of objects $(A,\ed{A}),(C,\ed{C})\inn\rwt{\CC}$, the map $\rwtGammad{((C,\ed{C}),(A,\ed{A}))}$ from \cref{def:completion-of-Gamma} is a homomorphism of abelian groups as $\Gammad{(C,A)}$ is one. It follows from the naturality of $\Gamma$ that $\rwt{\Gamma}$ is a natural transformation 
\mbox{$\rwt{\BE}(-,-)\Rightarrow\rwt{\BF}(\rwt{\SF}-,\rwt{\SF}-)$}.

Consider now an $\rwt{\BE}$-extension $(\ed{A},\alpha,\ed{C})\inn\rwt{\BE}((C,\ed{C}),(A,\ed{A}))$. Following the definition of $\rwt{\fs}$, we have \mbox{$\rwt{\fs}(\ed{A},\alpha,\ed{C})=[(\Xcom,\ed{\bullet})]$,} where $\fs(\alpha) = [\Xcom]$ and the idempotent \mbox{$\ed{\bullet}\colon \Xcom\to \Xcom$} is a lift of \mbox{$(\ed{A}, \ed{C}) \colon \alpha \to \alpha$}. Notice that \mbox{$\ft(\Gamma(\alpha)) = [\SFd{\com} \Xcom]$} as $(\SF,\Gamma)$ is $n$-exangulated. Moreover, the morphism \mbox{$\SFd{\com}\ed{\bullet}\colon \SFd{\com}\Xcom \to \SFd{\com}\Xcom$} is an idempotent lifting \mbox{$(\SF \ed{A}, \SF \ed{C})\colon \Gamma(\alpha)\to\Gamma(\alpha)$}. We hence see that $\rwt{\ft}(\rwt{\Gamma}(\ed{A},\alpha,\ed{C}))$ is given by the class
\[
[\begin{tikzcd}[column sep=2.6cm, ampersand replacement=\&]
    (\SF A,\SF \ed{A}) 
        \arrow{r}{\rwt{\SF}(\ed{1},\ed{1}\dd{0},\ed{A})} 
    \& (\SF \Xd{1},\SF \ed{1})
        \arrow{r}{\rwt{\SF}(\ed{2},\ed{2}\dd{1},\ed{1})} 
    \&\cdots 
        \arrow{r}{\rwt{\SF}(\ed{C},\ed{C}\dd{n},\ed{n})} 
    \& (\SF C,\SF \ed{C})
        \end{tikzcd}],
\]
which is $[\rwtSFd{\com}(\Xcom,\ed{\bullet})]$. This finishes the proof. 
\end{proof}

To consider $n$-exangulated categories as $0$-cells in a $2$-category, we use the notion of a morphism between $n$-exangulated functors. This is captured by the following definition.

\begin{defn}
\label{def:n-exangulated-natural-transformation}
(See \cite[Def.~4.1]{Bennett-TennenhausHauglandSandoyShah-the-category-of-extensions-and-a-characterisation-of-n-exangulated-functors}.) 
Suppose that \mbox{$(\SF,\Gamma), (\SG,\Lambda)\colon(\CC,\BE,\fs)\to(\CD,\BF,\ft)$} are $n$-exangulated functors. An \emph{$n$-exangulated natural transformation} \mbox{$(\SF,\Gamma)\Rightarrow(\SG,\Lambda)$} is a natural transformation $\Beth\colon\SF\Rightarrow\SG$ of additive functors such that, for all $A,C\inn\CC$ and each \mbox{$\alpha\inn\BE(C,A)$}, the pair $(\Bethd{A},\Bethd{C})$ satisfies 
\begin{equation}
\label{eqn:n-exangulated-natural-transformation-property}
    (\Bethd{A}\braced{\BF}  \Gamma(\alpha)
    = (\Bethd{C}\braceu{\BF}  \Lambda(\alpha).
\end{equation}
\end{defn}

Notice that equation \eqref{eqn:n-exangulated-natural-transformation-property} means that $(\Bethd{A},\Bethd{C})$ is a morphism \mbox{$\Gamma(\alpha) \to \Lambda(\alpha)$} of $\BF$-extensions.

For a natural transformation $\Beth\colon \SF\Rightarrow\SG$ of additive functors \mbox{$\SF,\SG\colon\CC\to\CD$}, recall that the completion \mbox{$\rwt{\Beth}\colon\rwt{\SF}\Rightarrow\rwt{\SG}$} is given by \mbox{$\rwtBethd{(X,e)} = (\SG e,(\SG e)\Bethd{X}\SF e,\SF e)$} for \mbox{$(X,e)\inn\rwt{\CC}$}. The proposition below shows that the completion of an $n$-exangulated natural transformation is again $n$-exangulated. 

\begin{lem}
\label{prop:completion-of-n-ex-naturaltr}
Suppose $\Beth \colon (\SF,\Gamma) \Rightarrow (\SG,\Lambda)$ is $n$-exangulated. Then $\rwt{\Beth}$ is an $n$-exangulated natural transformation $(\rwt{\SF},\rwt{\Gamma}) \Rightarrow (\rwt{\SG},\rwt{\Lambda})$. 
\end{lem}

\begin{proof}
Consider an $\rwt{\BE}$-extension $(\ed{A},\alpha,\ed{C})\inn\rwt{\BE}((C,\ed{C}),(A,\ed{A}))$. Using that \mbox{$(\ed{A}\braced{\BE}  \alpha=\alpha$}, we get \mbox{$(\SF \ed{A}\braced{\BF} \Gamma(\alpha)=\Gamma(\alpha)$} by the naturality of $\Gamma$. Similarly, we obtain \mbox{$\Lambda(\alpha)=(\SG \ed{C}\braceu{\BF} \Lambda(\alpha)$}. Since $\Beth$ is $n$-exangulated, we have \mbox{$(\Bethd{A}\braced{\BF}  \Gamma(\alpha)=(\Bethd{C}\braceu{\BF} \Lambda(\alpha)$}, while naturality of $\Beth$ yields \mbox{$(\SG \ed{C})\Bethd{C}=\Bethd{C}\SF \ed{C}$}. Combining these observations gives
\[
\begin{split}
(\SG \ed{A}\braced{\BF}  (\Bethd{A}\braced{\BF}  (\SF \ed{A}\braced{\BF}  \Gamma(\alpha)=
(\SG \ed{A}\braced{\BF}  (\Bethd{A}\braced{\BF}  \Gamma(\alpha)=
(\SG \ed{A}\braced{\BF}  (\Bethd{C}\braceu{\BF}  \Lambda(\alpha)=
(\Bethd{C}\braceu{\BF}  (\SG \ed{A}\braced{\BF}  \Lambda(\alpha)
\\
=
(\Bethd{C}\braceu{\BF}  \Lambda(\alpha)
=
(\Bethd{C}\braceu{\BF}  (\SG \ed{C}\braceu{\BF}  \Lambda(\alpha)
=
(\Bethd{C}\SF \ed{C}\braceu{\BF}  \Lambda(\alpha)
=
(\Bethd{C}\SF \ed{C}\braceu{\BF}  (\SG \ed{C}\braceu{\BF}  \Lambda(\alpha).
\end{split}
\]
Hence, we have that 
\begin{align*}
(\rwtBethd{(A,\ed{A})}\braced{\rwt{\BF}}\rwt{\Gamma}(\ed{A},\alpha,\ed{C})
&
= 
(\SG \ed{A},(\SG \ed{A})\Bethd{A}\SF \ed{A},\SF \ed{A}\braced{\rwt{\BF}}
	(\SF \ed{A},\Gamma(\alpha),\SF \ed{C}) && \text{(definitions of $\rwt{\Beth}, \rwt{\Gamma}$)}\\
	&= (\SG \ed{A},(\SG \ed{A}\braced{\BF}  (\Bethd{A}\braced{\BF}  (\SF \ed{A}\braced{\BF}  \Gamma(\alpha),\SF \ed{C}) &&\text{(\cref{def:bifunctor-on-IC-of-exangulated-category})}\\
	&= (\SG \ed{A},((\SG \ed{C}) \Bethd{C}\SF \ed{C}\braceu{\BF}  \Lambda(\alpha),\SF \ed{C})&&\text{(as above)}\\
	&= (\rwtBethd{(C,\ed{C})}\braceu{\rwt{\BF}} \, \rwt{\Lambda}(\ed{A},\alpha,\ed{C})&&\text{(definitions of $\rwt{\Beth}, \rwt{\Lambda}$),}
\end{align*}
as required.
\end{proof}

We now introduce $n$-exangulated analogues of the collections described in Notation~\ref{not:Exactcat-definition}.

\begin{notn}
\label{not:n-Exang-definition}
Following \cite[Not.~4.14]{Bennett-TennenhausHauglandSandoyShah-the-category-of-extensions-and-a-characterisation-of-n-exangulated-functors}, we write $\Exang{n}$ for the collection of $0$-cells, $1$-cells and $2$-cells consisting of $n$-exangulated categories, $n$-exangulated functors and $n$-exangulated natural transformations between these functors, respectively. Restricting $0$-cells in $\Exang{n}$ to small $n$-exangulated categories yields the $2$-category $\exang{n}$; see \cite[Cor.~4.15]{Bennett-TennenhausHauglandSandoyShah-the-category-of-extensions-and-a-characterisation-of-n-exangulated-functors}. We furthermore write $\ICExang{n}$ and $\ICexang{n}$ when only considering idempotent complete $0$-cells in $\Exang{n}$ and $\exang{n}$, respectively, and note that also $\ICexang{n}$ is a $2$-category. As before, we use a subscript $i \inn \{0,1,2\}$ to denote $i$-cells in the collections described above. 
\end{notn}

We conclude this section by constructing the $2$-functor $\nexangcomplete$ used in \cref{thm-C} in \cref{sec:introduction}. 

\begin{defn}
\label{def:exangcomplete}
Let $\nexangcomplete = (\nexangcompletecell{0},\nexangcompletecell{1},\nexangcompletecell{2}) \colon \Exang{n} \to \ICExang{n}$ be defined by the assignments $\nexangcompletecell{i}\colon \Exangcell{n}{i} \to \ICExangcell{n}{i}$, where:
\[
\nexangcompletecell{0}  (\CC,\BE,\fs)  
    \deff (\rwt{\CC},\rwt{\BE},\rwt{\fs}\mspace{1mu}),
\hspace{1cm}
\nexangcompletecell{1}  (\SF,\Gamma)  
    \deff (\rwt{\SF},\rwt{\Gamma}),
\hspace{1cm}
\nexangcompletecell{2} ( \Beth ) 
    \deff \rwt{\Beth}.
\]
\end{defn}

The result below is an $n$-exangulated analogue of \cref{thm:2-functor-of-exact-categories}. 

\begin{thm}
\label{thm:2-functor-of-n-exangulated-categories}
The following statements hold for the assignments $\nexangcompletecell{0}$, $\nexangcompletecell{1}$ and $\nexangcompletecell{2}$.

\begin{enumerate}[label=\textup{(\roman*)}] 
    \item\label{item:exangcomplete-0-and-1-give-functor-on-underyling-1-categories}
    The pair $(\nexangcompletecell{0}, \nexangcompletecell{1})$ defines a functor $\Exang{n}\to\ICExang{n}$.
    \item\label{item:exangcomplete-1-and-2-give-functor} 
    The pair $(\nexangcompletecell{1}, \nexangcompletecell{2})$ defines a functor 
    \[
    \Exang{n}((\CC,\BE,\fs), (\CD,\BF,\ft)) \to \ICExang{n}((\rwt{\CC},\rwt{\BE},\rwt{\fs}\mspace{1mu}), (\rwt{\CD},\rwt{\BF},\rwt{\ft}\mspace{1mu}))
    \]
    whenever $(\CC,\BE,\fs)$ and $(\CD,\BF,\ft)$ are $n$-exangulated categories.
    \item\label{item:exangcomplete2-preserves-horizontal-composition} 
    The assignment $\nexangcompletecell{2}$ preserves horizontal composition. 
\end{enumerate}
In particular, restricting $\nexangcomplete$ to small  categories yields a $2$-functor $\exang{n} \to \ICexang{n}$. 
\end{thm}

\begin{proof}
It follows from the discussion and results above, in particular Lemmas~\ref{prop:completion-of-n-ex-functor} and \ref{prop:completion-of-n-ex-naturaltr}, that the assignments are well-defined. 
The rest of this proof is similar to the proof of \cref{thm:2-functor-of-exact-categories}. 
Functoriality in \ref{item:exangcomplete-0-and-1-give-functor-on-underyling-1-categories} is straightforward to check. 
For \ref{item:exangcomplete2-preserves-horizontal-composition}, note that $n$-exangulated natural transformations are closed under horizontal composition by \cite[Prop.~4.8]{Bennett-TennenhausHauglandSandoyShah-the-category-of-extensions-and-a-characterisation-of-n-exangulated-functors}, and then again apply \cite[Rem.~6.8]{Buhler-exact-categories}. 
Lastly, for \ref{item:exangcomplete-1-and-2-give-functor}, notice first that \mbox{$\Exang{n}((\CC,\BE,\fs), (\CD,\BF,\ft))$}, and hence also \mbox{$\ICExang{n}((\rwt{\CC},\rwt{\BE},\rwt{\fs}\mspace{1mu}), (\rwt{\CD},\rwt{\BF},\rwt{\ft}\mspace{1mu}))$}, is indeed a category by \cite[Prop.~4.12]{Bennett-TennenhausHauglandSandoyShah-the-category-of-extensions-and-a-characterisation-of-n-exangulated-functors}. 
Part \ref{item:exangcomplete-1-and-2-give-functor} then follows from \cite[Rem.~6.8]{Buhler-exact-categories} and noting that, for an $n$-exangulated functor $(\SF,\Gamma)$, the identity $\id{(\SF,\Gamma)}$ is just $\id{\SF}$ (see \cite[Def.~4.3]{Bennett-TennenhausHauglandSandoyShah-the-category-of-extensions-and-a-characterisation-of-n-exangulated-functors}).
\end{proof}


\section{The category of extensions and idempotent completion}
\label{sec:cat-of-extensions-and-splitting-idempotents}

We assume throughout  \cref{sec:cat-of-extensions-and-splitting-idempotents} that $\CC$ is an additive category and that \mbox{$\BE \colon \CCu{\op}\times\CC \to \Ab$} is a biadditive functor. 
Note in particular that any $n$-exangulated category $(\CC,\BE,\fs)$ gives rise to such a pair $(\CC,\BE)$ by ignoring the realisation $\fs$. In Subsection~\ref{subsec:cat-of-extensions} we  recall the definition of the category $\catext{\BE}{\CC}$ of extensions associated to  $(\CC,\BE)$, before proving \cref{prop-A} from \cref{sec:introduction}. Building on this result, our ultimate goal is to show that, for the  biadditive functor \mbox{$\rwt{\BE}\colon \rwt{\CC}^{\op}\times\rwt{\CC} \to \Ab$} from \cref{def:bifunctor-on-IC-of-exangulated-category}, the category $\catext{\rwt{\BE}}{\rwt{\CC}}$  of extensions of the idempotent completion  is equivalent to the idempotent completion $\rwt{\catext{\BE}{\CC}}$ of $\catext{\BE}{\CC}$. 
These two categories are described explicitly in Subsection~\ref{subsec:cat-of-extensions-of-idempotent-completion} and Subsection~\ref{subsec:idempotent-completion-of-cat-of-extensions}, respectively, culminating in a proof of \cref{thm-B} from \cref{sec:introduction}. In \cref{example-matrices} we provide an algebraic example exhibiting an application of  \cref{thm:cat-of-extensions-of-IC-isomorphic-to-IC-of-cat-of-extensions}. 


\subsection{The category of extensions}
\label{subsec:cat-of-extensions}

The category of extensions associated to $(\CC,\BE)$ is denoted by $\catext{\BE}{\CC}$. The objects of $\catext{\BE}{\CC}$ are $\BE$-extensions, and the morphisms are morphisms of $\BE$-extensions. Recall from \cref{sec:idempotent-completion-n-exangulated-cats} that this means that an object is an element $\alpha\inn\BE(C,A)$ for some objects $A,C\inn\CC$, while a morphism from $\alpha \inn \BE(C,A)$ to $\beta \inn \BE(D,B)$ is given by a pair $(a,c)$ of morphisms $a\colon A\to B$ and $c\colon C\to D$ in $\CC$ satisfying $\ad{\BE}\alpha = \cu{\BE}\beta$. 

As shown in \cite[Prop.~3.2]{Bennett-TennenhausHauglandSandoyShah-the-category-of-extensions-and-a-characterisation-of-n-exangulated-functors}, one can define an exact structure $\CXd{\BE}$ on  $\catext{\BE}{\CC}$ as follows. Let $\alpha\inn\BE(C,A)$, $\beta\inn\BE(D,B)$ and $\gamma\inn\BE(G,E)$ be objects in $\catext{\BE}{\CC}$. A sequence 
\begin{equation*}
\label{eqn:conflation-in-exact-structure}
\begin{tikzcd}
\alpha \arrow{r}{(a,c)}
    & \beta \arrow{r}{(b,d)}
    & \gamma
\end{tikzcd}
\end{equation*}
of composable morphisms in $\catext{\BE}{\CC}$ lies in the class $\CXd{\BE}$ if and only if the morphisms $a$ and $c$ in $\CC$ are both sections with $b=\cok a$ and $d=\cok c$.

By definition, a category is idempotent complete provided any idempotent endomorphism splits. Note also that a morphism of $\BE$-extensions $(\ed{A},\ed{C})\colon \alpha\to\alpha$ for $\alpha \inn \BE(C,A)$ is an idempotent in $\catext{\BE}{\CC}$ if and only if both $\ed{A}\colon A \to A$ and $\ed{C} \colon C \to C$ are idempotents in $\CC$. Hence, \cref{prop-A} follows from \cref{prop:idempotents-split-category-of-extensions}. The authors are grateful to Dixy Msapato for pointing out \cite[Lem.~3.23]{Msapato-the-karoubi-envelope-and-weak-idempotent-completion-of-an-extriangulated-category}, which motivated the proof of the result below. 

\begin{prop}
\label{prop:idempotents-split-category-of-extensions} 
Let $\alpha\inn\BE(C,A)$ and  suppose $(\ed{A},\ed{C})\inn \Endd{\catext{\BE}{\CC}}(\alpha)$ is an idempotent. Then $\ed{A}$ and $\ed{C}$ split in $\CC$ if and only if $(\ed{A},\ed{C})$ splits in $\catext{\BE}{\CC}$. 
\end{prop}

\begin{proof}
$(\Rightarrow)$\;\;
Assume that $\ed{A}$ and $\ed{C}$ split in $\CC$. As $\ed{A}$ splits, there exist morphisms $r \colon A \to B$ and $s \colon B \to A$ such that $sr=\ed{A}$ and $rs=\id{B}$. Similarly, there exist $u \colon C \to D$ and $v \colon D \to C$ with $vu=\ed{C}$ and $uv=\id{D}$, because $\ed{C}$ splits. Consider the $\BE$-extension $\rd{\BE}\vu{\BE}\alpha\inn\BE(D,B)$. We see that $(s,v)\colon \rd{\BE}\vu{\BE}\alpha \to \alpha$ is a morphism in $\catext{\BE}{\CC}$, as
\[
\ss{\BE}(\rd{\BE}\vu{\BE}\alpha) = \vu{\BE}(sr\braced{\BE}  \alpha = \vu{\BE}(e_A\braced{\BE}  \alpha = \vu{\BE}(e_C\braceu{\BE}  \alpha = \vu{\BE}(vu\braceu{\BE}  \alpha = (vuv\braceu{\BE}  \alpha = \vu{\BE}\alpha.
\]
Analogously, one can show that $(r,u)\colon \alpha \to \rd{\BE}\vu{\BE}\alpha$ is a morphism of $\BE$-extensions. Notice that $(s,v)\circ(r,u) = (\ed{A},\ed{C})$ and $(r,u)\circ(s,v) = (\id{B},\id{D})$, which is the identity on $\rd{\BE}\vu{\BE}\alpha$. This gives a splitting of $(\ed{A},\ed{C})$, as required.  

$(\Leftarrow)$\;\; 
If $(\ed{A},\ed{C})$ splits, then there exist $\beta \inn \BE(D,B)$ and  morphisms $(r,u)\colon\alpha\to\beta$ and $(s,v)\colon\beta\to\alpha$ in $\catext{\BE}{\CC}$ such that $(s,v)\circ(r,u) = (\ed{A},\ed{C})$ and \mbox{$(r,u)\circ(s,v) = (\id{B},\id{D})$}. These equations yield splittings of $\ed{A}$ and $\ed{C}$ in $\CC$.
\end{proof}

We finish this subsection by deducing \cref{cor-KRS-property}, showing that given certain assumptions on $\CC$, the Krull--Remak--Schmidt property for $\CC$ implies the same property for $\catext{\BE}{\CC}$. In order to see this, we first recall some terminology.

For the rest of \Cref{subsec:cat-of-extensions}, let $\ring$ be a commutative ring. The additive category $\CC$ is said to be \emph{$\ring$-linear} if $\CC(X,Y)$ is an $\ring$-module for all $X,Y\inn \CC$, and we have \mbox{$(\lambda g)f=\lambda (gf)=g(\lambda f)$} for all $\lambda \inn \ring$ and all composable morphisms $f$ and $g$ in $\CC$. When $\CC$ is $\ring$-linear, the bifunctor $\BE$ is $\ring$-\emph{bilinear} provided that each abelian group $\BE(C,A)$ has the structure of an $\ring$-module, and we have \mbox{$\BE(\lambda c,a)=\lambda\BE(c,a)=\BE(c,\lambda a)$} for all $\lambda \inn \ring$ and any morphisms $a$ and $c$ in $\CC$. Recall that an $\ring$-linear category $\CC$ is called \emph{$\Hom$-finite (over $\ring$)} if each $\ring$-module $\CC(X,Y)$ has  finite length (see e.g.\ \cite[Sec.~5]{Krause-KS-cats-and-projective-covers}).  

\begin{prop}
\label{lem-k-linear-hom-finite-preserved-by-cat-of-exts}   
If $\CC$ is $\ring$-linear and $\BE$ is $\ring$-bilinear, then $\catext{\BE}{\CC}$ is also $\ring$-linear. If in addition $\CC$ is $\Hom$-finite, then so is $\catext{\BE}{\CC}$. 
\end{prop}

\begin{proof}
Fix objects $\alpha$ and $\beta$ in $\catext{\BE}{\CC}$, say where $\alpha\inn\BE(C,A)$ and $\beta\inn \BE(D,B)$. Given a morphism \mbox{$(a,c)\colon \alpha\to\beta$} in $\catext{\BE}{\CC}$, we have 
\[
(\lambda a\braced{\BE}  \alpha=
\BE(\id{C},\lambda a)(\alpha)=
\lambda \BE(\id{C},a)(\alpha)=
\lambda \BE(c,\id{A})(\beta)=
\BE(\lambda c,\id{A})(\beta)
=(\lambda c\braceu{\BE}  \beta
\]
as $\BE$ is $\ring$-bilinear. This means that $(\lambda a,\lambda c)\colon \alpha\to \beta$ is a morphism in $\catext{\BE}{\CC}$, and we take this to be the action of $\lambda$ on $(a,c)$. Using that $\CC$ is $\ring$-linear, it is straightforward to check that the $\Hom$-sets of $\catext{\BE}{\CC}$ are $\ring$-modules under this multiplication. Consider another morphism \mbox{$(b,d)\colon \beta \to \gamma$} in $\catext{\BE}{\CC}$. Since $\CC$ is $\ring$-linear, we have
\[
(b,d)(\lambda a,\lambda c)=(b(\lambda a),d(\lambda c))=(\lambda (ba),\lambda(dc))=((\lambda b)a,(\lambda d), c)=(\lambda b,\lambda d)(a,c),
\]
which is the action of $\lambda$ on $(b,d)\circ (a,c)$. This proves that $\catext{\BE}{\CC}$ is $\ring$-linear. 

The arguments above show that the collection of morphisms $\alpha\to \beta$ in $\catext{\BE}{\CC}$ defines an $\ring$-submodule of the direct sum $\CC(A,B)\oplus \CC(C,D)$. The length of this submodule is bounded above by the sum of the lengths of $\CC(A,B)$ and $\CC(C,D)$, proving the second assertion. 
\end{proof}

Recall that $\CC$ is said to be a \emph{Krull--Schmidt category} if every object decomposes into a finite direct sum of objects with local endomorphism rings. 

\begin{cor}
\label{cor-KRS-property}
Suppose that $\CC$ is $\ring$-linear, $\Hom$-finite and Krull--Schmidt, and that $\BE$ is $\ring$-bilinear. Then $\catext{\BE}{\CC}$ is also $\ring$-linear, $\Hom$-finite and Krull--Schmidt.
\end{cor}

\begin{proof}
Note that an $R$-linear $\Hom$-finite category  is Krull--Schmidt if and only if it is idempotent complete; see e.g.\ \cite[Cor.~A.2]{ChenYeZhang-Algebras-of-derived-dimension-zero} or \cite[Thm.~6.1]{Shah-Krull-Remak-Schmidt}. Using this, the result now follows by combining \cref{prop-A} and \cref{lem-k-linear-hom-finite-preserved-by-cat-of-exts}.
\end{proof}


\subsection{The category of extensions of the idempotent completion}
\label{subsec:cat-of-extensions-of-idempotent-completion}

As recalled in \cref{def:bifunctor-on-IC-of-exangulated-category}, there is a biadditive functor
$\rwt{\BE}\colon \rwt{\CC}^{\op}\times\rwt{\CC} \to \Ab$  defined canonically from $\BE$ where $\rwt{\CC}$ is the idempotent completion of $\CC$. As in Subsection~\ref{subsec:cat-of-extensions} 
we can consider the category $\catext{\rwt{\BE}}{\rwt{\CC}}$ of extensions associated to $(\rwt{\CC},\rwt{\BE})$. We give an explicit description of $\catext{\rwt{\BE}}{\rwt{\CC}}$ and the exact structure $\CXd{\rwt{\BE}}$ below.\medskip

\noindent\underline{\smash{Objects}}: The objects of $\catext{\rwt{\BE}}{\rwt{\CC}}$ are of the form $(\ed{A},\alpha,\ed{C})\inn\rwt{\BE}((C,\ed{C}),(A,\ed{A}))$ for $(C,\ed{C})$ and $(A,\ed{A})$ in $\rwt{\CC}$. In particular, the morphisms \mbox{$\ed{A} \colon A \to A$} and \mbox{$\ed{C} \colon C \to C$} in $\CC$ are idempotents,  and $\alpha\inn\BE(C,A)$ is an $\BE$-extension satisfying \mbox{$(\ed{A}\braced{\BE}  \alpha  = \alpha = (\ed{C}\braceu{\BE} \alpha$}.\medskip

\noindent\underline{\smash{Morphisms}}: A morphism 
$(\ed{A},\alpha, \ed{C})
    \to (\ed{B},\beta, \ed{D})
$
in $\catext{\rwt{\BE}}{\rwt{\CC}}$ is a pair \mbox{$((\ed{B},a,\ed{A}),(\ed{D},c,\ed{C}))$}, 
where 
\mbox{$(\ed{B},a,\ed{A})\colon (A,\ed{A}) \to (B,\ed{B})$} 
and 
\mbox{$(\ed{D},c,\ed{C})\colon (C,\ed{C}) \to (D,\ed{D})$} are morphisms in $\rwt{\CC}$ and 
\mbox{$(\ed{B},a,\ed{A}\braced{\rwt{\BE}}(\ed{A},\alpha, \ed{C}) 
    = (\ed{D},c,\ed{B}\braceu{\rwt{\BE}}(\ed{B},\beta, \ed{D})$}. 
This means that \mbox{$(a,c)\colon \alpha \to \beta$} is a morphism of $\BE$-extensions, \mbox{$a\ed{A} = a = \ed{B}a$} and \mbox{$c\ed{C}=c=\ed{D}c$}.\medskip

\noindent\underline{\smash{Composition}}: Composition in $\catext{\rwt{\BE}}{\rwt{\CC}}$ is defined component-wise. Explicitly, the composition of \mbox{$((\ed{B},a,\ed{A}),(\ed{D},c,\ed{C}))$} and \mbox{$((\ed{E},b,\ed{B}),(\ed{G},d,\ed{D}))$} is \mbox{$((\ed{E},ba,\ed{A}),(\ed{G},dc,\ed{C}))$}.\medskip

\noindent\underline{\smash{Identity morphisms}}: The identity on $(\ed{A},\alpha,\ed{C})$ in $\catext{\rwt{\BE}}{\rwt{\CC}}$ is \mbox{$((\ed{A},\ed{A},\ed{A}),(\ed{C},\ed{C},\ed{C}))$}.\medskip

\noindent\underline{\smash{Preadditivity}}: The addition of morphisms is component-wise. Explicitly, the addition of \mbox{$((\ed{B},a,\ed{A}),(\ed{D},c,\ed{C}))$} and \mbox{$((\ed{B},a',\ed{A}),(\ed{D},c',\ed{C}))$} is \mbox{$((\ed{B},a+a',\ed{A}),(\ed{D},c+c',\ed{C}))$}.\medskip

\noindent\underline{Exact structure}:
The collection $\CXd{\rwt{\BE}}$ consists of kernel-cokernel pairs
\[
\begin{tikzcd}[column sep=3.7cm]
(\ed{A}, \alpha, \ed{C})
	\arrow{r}{((\ed{B},a,\ed{A}),(\ed{D},c,\ed{C}))}
&(\ed{B}, \beta, \ed{D})
	\arrow{r}{((\ed{E},b,\ed{B}),(\ed{G},d,\ed{D}))} 
&(\ed{E}, \gamma, \ed{G})
\end{tikzcd}
\]
in $\catext{\rwt{\BE}}{\rwt{\CC}}$ such that $(\ed{B},a,\ed{A})$ and $(\ed{D},c,\ed{C})$ are sections with \mbox{$(\ed{E},b,\ed{B}) = \cok(\ed{B},a,\ed{A})$} and \mbox{$(\ed{G},d,\ed{D}) = \cok(\ed{D},c,\ed{C})$}. These conditions are equivalent to the sequences
\begin{align*}
    \begin{tikzcd}[column sep=2cm,ampersand replacement=\&]
		(A,\ed{A}) 
			\arrow{r}{(\ed{B},a,\ed{A})}
		\& (B,\ed{B}) 
			\arrow{r}{(\ed{E},b,\ed{B})}
		\& (E,\ed{E}), 
    \end{tikzcd} \\
    \begin{tikzcd}[column sep=2cm,ampersand replacement=\&]
		(C,\ed{C}) 
			\arrow{r}{(\ed{D},c,\ed{C})}
		\& (D,\ed{D}) 
			\arrow{r}{(\ed{G},d,\ed{D})}
		\& (G,\ed{G})\phantom{,} 
    \end{tikzcd}
\end{align*}
being split exact in $\rwt{\CC}$.\medskip

We have an immediate corollary of \cref{prop-A}.

\begin{cor}
\label{cor:catex-of-tilde-C-is-IC} The exact category $(\catext{\rwt{\BE}}{\rwt{\CC}}, \CXd{\rwt{\BE}})$ is idempotent complete.
\end{cor}
	

\subsection{The idempotent completion of the category of extensions}
\label{subsec:idempotent-completion-of-cat-of-extensions}

In contrast to what is done in Subsection~\ref{subsec:cat-of-extensions-of-idempotent-completion}, we may first consider the category of extensions associated to $(\CC,\BE)$, which forms part of an exact category $(\catext{\BE}{\CC},\CXd{\BE})$. Then we may take the idempotent completion, resulting in an idempotent complete exact category that we denote by \mbox{$(\rwt{\catext{\BE}{\CC}},\rwt{\CXd{\BE}})$}; see \cref{prop:inclusion-functor-into-Karoubi-envelope-is-ff-exact}. We proceed with an explicit description of $\rwt{\catext{\BE}{\CC}}$ and the exact structure $\rwt{\CXd{\BE}}$.\medskip

\noindent\underline{\smash{Objects}}: 
The objects of $\rwt{\catext{\BE}{\CC}}$ are of the form $(\alpha, (\ed{A},\ed{C}))$, where $\alpha\inn\BE(C,A)$ and \mbox{$(\ed{A},\ed{C})\colon \alpha \to \alpha$} is an idempotent morphism of $\BE$-extensions. This means that \mbox{$e_A \colon A \to A$} and \mbox{$\ed{C} \colon C \to C$} are idempotents in $\CC$ and that \mbox{$(\ed{A}\braced{\BE}  \alpha  = (\ed{C}\braceu{\BE} \alpha$}.\medskip

\noindent\underline{\smash{Morphisms}}: A morphism $(\alpha, (\ed{A},\ed{C})) \to (\beta, (\ed{B},\ed{D}))$ in $\rwt{\catext{\BE}{\CC}}$ is given by a triple \mbox{$((\ed{B},\ed{D}),(a,c),(\ed{A},\ed{C}))$}, where 
$(a,c) \colon \alpha \to \beta$ is a morphism of $\BE$-extensions and we have \mbox{$(a,c)(\ed{A},\ed{C}) = (a,c) = (\ed{B},\ed{D})(a,c)$}.\medskip

\noindent\underline{\smash{Composition}}: 
The composition of two composable morphisms \mbox{$((\ed{B},\ed{D}),(a,c),(\ed{A},\ed{C}))$} and \mbox{$((\ed{E},\ed{G}),(b,d),(\ed{B},\ed{D}))$} in $\rwt{\catext{\BE}{\CC}}$ is $((\ed{E},\ed{G}),(ba,dc),(\ed{A},\ed{C}))$.
\medskip

\noindent\underline{\smash{Identity morphisms}}: 
The identity on $(\alpha, (\ed{A},\ed{C}))$ in $\rwt{\catext{\BE}{\CC}}$ is \mbox{$((\ed{A},\ed{C}),(\ed{A},\ed{C}),(\ed{A},\ed{C}))$}.\medskip

\noindent\underline{\smash{Preadditivity}}: 
Let $((\ed{B},\ed{D}),(a,c),(\ed{A},\ed{C}))$ and $((\ed{B},\ed{D}),(a',c'),(\ed{A},\ed{C}))$ be morphisms from $(\alpha,(e_A,e_C))$ to $(\beta,(e_B,e_D))$ in $\rwt{\catext{\BE}{\CC}}$. The addition of these two morphisms is given by \mbox{$((\ed{B},\ed{D}),(a+a',c+c'),(\ed{A},\ed{C}))$}.
\medskip

\noindent\underline{Exact structure}:
The elements in $\rwt{\CXd{\BE}}$ are direct summands of images of elements in $\CXd{\BE}$ under the functor \mbox{$\SId{\catext{\BE}{\CC}}\colon \catext{\BE}{\CC} \to \rwt{\catext{\BE}{\CC}}$}; see the discussion before \cref{prop:inclusion-functor-into-Karoubi-envelope-is-ff-exact}. In other words, they are direct summands of kernel-cokernel pairs in $\rwt{\catext{\BE}{\CC}}$ of the form 
\[
\begin{tikzcd}[column sep=1cm]
(\alpha, (\id{A},\id{C}))
	\arrow{r}{p}
& (\beta, (\id{B},\id{D}))
	\arrow{r}{q}
& (\gamma, (\id{E},\id{G})), 
\end{tikzcd}
\]
where 
\[
\begin{array}{cc}
p \deff ((\id{B},\id{D}),(a,c),(\id{A},\id{C})), 
&
q \deff ((\id{E},\id{G}),(b,d),(\id{B},\id{D}))
\end{array}
\]
and
$
\begin{tikzcd}
\alpha \arrow{r}{(a,c)}
    & \beta \arrow{r}{(b,d)}
    & \gamma
\end{tikzcd}
$
is an element of $\CXd{\BE}$. 

\begin{warning}
\label{warning:categories-not-isomorphic}
Even if $(\alpha,(\ed{A},\ed{C}))$ is an object in $\rwt{\catext{\BE}{\CC}}$, the $\BE$-extension $\alpha\inn\BE(C,A)$ does \emph{not} necessarily satisfy $(\ed{A}\braced{\BE}  \alpha = \alpha$ or $(\ed{C}\braceu{\BE}  \alpha =\alpha$. For example, if $\alpha \neq 0$, then $(\alpha,(0,0))$ is an object in $\rwt{\catext{\BE}{\CC}}$, but  \mbox{$\zerod{\BE}\alpha = \zerou{\BE}\alpha = 0  \neq \alpha$}. In particular, this means that the objects of $\rwt{\catext{\BE}{\CC}}$ are not canonically in one-to-one correspondence with those of $\catext{\rwt{\BE}}{\rwt{\CC}}$.
\end{warning}

\cref{warning:categories-not-isomorphic} tells us that we cannot expect the categories $\catext{\rwt{\BE}}{\rwt{\CC}}$ and $\rwt{\catext{\BE}{\CC}}$ to be \emph{isomorphic} in general. Despite this, we prove that they are always equivalent. In the following, we use the Hebrew letters $\Mem$ (mem), $\Shin$ (shin) and $\Tsadi$ (tsadi). Note that the functor $\Shind{(\CC,\BE)}$ that we define in the proof of \cref{thm:cat-of-extensions-of-IC-isomorphic-to-IC-of-cat-of-extensions} below will be used to construct a natural transformation in \cref{sec:2-categorical-perspective}, which is the reason for our choice of notation.

\begin{thm}
\label{thm:cat-of-extensions-of-IC-isomorphic-to-IC-of-cat-of-extensions}
The exact categories $(\catext{\rwt{\BE}}{\rwt{\CC}}, \CXd{\rwt{\BE}})$ and $(\rwt{\catext{\BE}{\CC}}, \rwt{\CXd{\BE}})$ are equivalent. 
\end{thm}

\begin{proof}
We establish an exact functor
$\Shind{(\CC,\BE)} \colon 
        (\catext{\rwt{\BE}}{\rwt{\CC}}, \CXd{\rwt{\BE}})
            \to 
	(\rwt{\catext{\BE}{\CC}}, \rwt{\CXd{\BE}})
$
and an exact quasi-inverse 
$\Tsadid{(\CC,\BE)} \colon 
	(\rwt{\catext{\BE}{\CC}}, \rwt{\CXd{\BE}})
		\to 
	(\catext{\rwt{\BE}}{\rwt{\CC}}, \CXd{\rwt{\BE}})
$.

Define $\Shind{(\CC,\BE)}$ by
\[
\Shind{(\CC,\BE)}(\ed{A},\alpha,\ed{C}) 
    \deff (\alpha,(\ed{A},\ed{C}))
\]
on objects and 
\[
\Shind{(\CC,\BE)}((\ed{B},a,\ed{A}),(\ed{D},c,\ed{C})) 
	\deff ((\ed{B},\ed{D}),(a,c),(\ed{A},\ed{C}))
\]
on morphisms. By our explicit description of $\catext{\rwt{\BE}}{\rwt{\CC}}$ and $\rwt{\catext{\BE}{\CC}}$ in Subsection~\ref{subsec:cat-of-extensions-of-idempotent-completion} and Subsection~\ref{subsec:idempotent-completion-of-cat-of-extensions}, respectively, we see that $\Shind{(\CC,\BE)}$ is a well-defined additive functor.

Define $\Tsadid{(\CC,\BE)}$ by 
\[
\Tsadid{(\CC,\BE)} (\alpha,(\ed{A},\ed{C})) 
    \deff (\ed{A},(\ed{A}\braced{\BE}  \alpha,\ed{C})
\]
on objects and 
\[
\Tsadid{(\CC,\BE)} ((\ed{B},\ed{D}),(a,c),(\ed{A},\ed{C})) 
    \deff ((\ed{B},a,\ed{A}),(\ed{D},c,\ed{C}))
\] 
on morphisms. Note that $\Tsadid{(\CC,\BE)}$ is well-defined on objects, since
\[
(\ed{C}\braceu{\BE}  (\ed{A}\braced{\BE}  \alpha 
    = (\ed{A}\braced{\BE}  (\ed{C}\braceu{\BE}  \alpha 
    = (\ed{A}\braced{\BE}  (\ed{A}\braced{\BE}  \alpha 
    = (\ed{A}\braced{\BE}  \alpha.
\]
It is straightforward to check that $\Tsadid{(\CC,\BE)}$ is well-defined on morphisms, and that it is an additive functor.

The composite $\Tsadid{(\CC,\BE)}\circ\Shind{(\CC,\BE)}$ is the identity functor $\id{\catext{\rwt{\BE}}{\rwt{\CC}}}$ of $\catext{\rwt{\BE}}{\rwt{\CC}}$, as \mbox{$(\ed{A}\braced{\BE}  \alpha = \alpha$} whenever $(\ed{A},\alpha,\ed{C})\inn\catext{\rwt{\BE}}{\rwt{\CC}}$. For each object $(\alpha,(\ed{A},\ed{C}))$ in $\rwt{\catext{\BE}{\CC}}$, set
\[
\Memd{(\alpha,(\ed{A},\ed{C}))} 
	\deff ((\ed{A},\ed{C}),(\ed{A},\ed{C}),(\ed{A},\ed{C}))
	\colon 
	(\alpha,(\ed{A},\ed{C}))
		\to 
	((\ed{A}\braced{\BE}  \alpha,(\ed{A},\ed{C})).
\] 
This is an isomorphism in $\rwt{\catext{\BE}{\CC}}$. Checking that 
\[
\Mem 
	\deff \Setd{\Memd{(\alpha,(\ed{A},\ed{C}))}}{(\alpha,(\ed{A},\ed{C}))\inn\rwt{\catext{\BE}{\CC}}}
	\colon 
	\id{\rwt{\catext{\BE}{\CC}}} 
		\Longrightarrow 
	\Shind{(\CC,\BE)}\circ\Tsadid{(\CC,\BE)}
\]
is natural is straightforward. 

It remains to show that $\Shind{(\CC,\BE)}$ and $\Tsadid{(\CC,\BE)}$ are exact functors. Recall that the direct sum of two objects $(X,\ed{X})$ and $(Y,\ed{Y})$ in $\rwt{\CC}$ is given by 
\[
(X,\ed{X})\oplus (Y,\ed{Y}) = (X\oplus Y,\ed{X}\oplus \ed{Y}) = \left(X\oplus Y,\begin{psmallmatrix}\ed{X} \amph 0 \\ 0 \amph \ed{Y}\end{psmallmatrix}\right).
\]

We first check that $\Shind{(\CC,\BE)}$ is exact. Let \begin{equation}\label{eqn:Shin-exact-original-seq}
\begin{tikzcd}[column sep=4cm]
(\ed{A}, \alpha, \ed{C})
	\arrow{r}{((\ed{B},a,\ed{A}),(\ed{D},c,\ed{C}))}
&(\ed{B}, \beta, \ed{D})
	\arrow{r}{((\ed{E},b,\ed{B}),(\ed{G},d,\ed{D}))} 
&(\ed{E}, \gamma, \ed{G})
\end{tikzcd}
\end{equation}
be an arbitrary element of $\CXd{\rwt{\BE}}$. The underlying sequences of \eqref{eqn:Shin-exact-original-seq} are split exact in $\rwt{\CC}$, so we may without loss of generality assume that \mbox{$B = A\oplus E$}, \mbox{$\ed{B} = \ed{A} \oplus \ed{E}$}, \mbox{$D = C\oplus G$}, \mbox{$\ed{D} = \ed{C} \oplus \ed{G}$}, \mbox{$(a,c) = (\begin{psmallmatrix}\ed{A}\\ 0\end{psmallmatrix}, \begin{psmallmatrix}\ed{C}\\ 0\end{psmallmatrix})$} and \mbox{$(b,d) = (\begin{psmallmatrix}0\amph \ed{E}\end{psmallmatrix}, \begin{psmallmatrix}0 \amph \ed{G}\end{psmallmatrix})$}. Applying $\Shind{(\CC,\BE)}$ to \eqref{eqn:Shin-exact-original-seq} then yields the sequence
\begin{equation}
\label{eqn:Shin-applied-to-original-seq}
\begin{tikzcd}[column sep=1cm]
(\alpha, (\ed{A},\ed{C}))
	\arrow{r}{r}
& (\beta, (\ed{A}\oplus \ed{E},\ed{C}\oplus \ed{G}))
	\arrow{r}{s}
& (\gamma, (\ed{E},\ed{G})),
\end{tikzcd}
\end{equation}
where 
\begin{align*}
r &\deff ((\begin{psmallmatrix}\ed{A} \amph 0 \\ 0 \amph \ed{E}\end{psmallmatrix},\begin{psmallmatrix}\ed{C} \amph 0 \\ 0 \amph \ed{G}\end{psmallmatrix}),(\begin{psmallmatrix}\ed{A} \\ 0\end{psmallmatrix},\begin{psmallmatrix}\ed{C} \\ 0\end{psmallmatrix}),(\ed{A},\ed{C})),\\
s &\deff ((\ed{E},\ed{G}),(\begin{psmallmatrix}0\amph \ed{E}\end{psmallmatrix}, \begin{psmallmatrix}0 \amph \ed{G}\end{psmallmatrix}),(\begin{psmallmatrix}\ed{A} \amph 0 \\ 0 \amph \ed{E}\end{psmallmatrix},\begin{psmallmatrix}\ed{C} \amph 0 \\ 0 \amph \ed{G}\end{psmallmatrix})). 
\end{align*}
We claim that $\eqref{eqn:Shin-applied-to-original-seq}$ is a direct summand of the sequence
\begin{equation}
\label{eqn:the-seq-that-should-lie-in-tildeX}
\begin{tikzcd}[column sep=1cm]
(\alpha, (\id{A},\id{C}))
	\arrow{r}{t}
& (\beta, (\id{A\oplus E}, \id{C\oplus G}))
	\arrow{r}{u}
& (\gamma, (\id{E},\id{G})), 
\end{tikzcd}
\end{equation}
where 
\begin{align*}
t &\deff ((\id{A\oplus E}, \id{C\oplus G}),(\begin{psmallmatrix}\id{A} \\ 0\end{psmallmatrix}, \begin{psmallmatrix}\id{C} \\ 0\end{psmallmatrix}),(\id{A},\id{C})), \\
u &\deff ( (\id{E},\id{G}),(\begin{psmallmatrix}0 \amph \id{E}\end{psmallmatrix},\begin{psmallmatrix}0 \amph \id{G}\end{psmallmatrix}),(\id{A\oplus E}, \id{C\oplus G}) ) .
\end{align*}
Notice first that  $(\begin{psmallmatrix}\id{A} \\ 0\end{psmallmatrix}, \begin{psmallmatrix}\id{C} \\ 0\end{psmallmatrix})$ is a morphism $\alpha \to \beta$ of $\BE$-extensions, since 
\[
\tensor*[]{\begin{psmallmatrix}\id{A} \\ 0\end{psmallmatrix}}{_{\BE}} \alpha = \tensor*[]{\begin{psmallmatrix}\id{A} \\ 0\end{psmallmatrix}}{_{\BE}} (e_A\braced{\BE}   \alpha = \tensor*[]{\begin{psmallmatrix}\ed{A} \\ 0\end{psmallmatrix}}{_{\BE}} \alpha = \tensor*[]{\begin{psmallmatrix}\ed{C} \\ 0\end{psmallmatrix}}{^{\BE}} \beta = \tensor*[]{\begin{psmallmatrix}\id{C} \\ 0\end{psmallmatrix}}{^{\BE}} \tensor*[]{\begin{psmallmatrix}\ed{C} & 0 \\ 0 & \ed{G}\end{psmallmatrix}}{^{\BE}}\beta = \tensor*[]{\begin{psmallmatrix}\id{C} \\ 0\end{psmallmatrix}}{_{\BE}} \beta.
\]
Similarly, the pair  $(\begin{psmallmatrix}0 \amph \id{E}\end{psmallmatrix},\begin{psmallmatrix}0 \amph \id{G}\end{psmallmatrix})$ is a morphism $\beta \to \gamma$ in $\catext{\BE}{\CC}$. In particular, this yields that \eqref{eqn:the-seq-that-should-lie-in-tildeX} is indeed a sequence in $\rwt{\catext{\BE}{\CC}}$. To verify that \eqref{eqn:Shin-applied-to-original-seq} is a direct summand of \eqref{eqn:the-seq-that-should-lie-in-tildeX}, notice that there is a section induced by the morphisms \mbox{$((\id{A},\id{C}),(\ed{A},\ed{C}),(\ed{A},\ed{C}))$}, \mbox{$((\id{A\oplus E}, \id{C \oplus G}),(\ed{A}\oplus \ed{E},\ed{C}\oplus \ed{G}),(\ed{A}\oplus \ed{E},\ed{C}\oplus \ed{G}))$} and \mbox{$((\id{E}, \id{G}),(\ed{E},\ed{G}),(\ed{E},\ed{G}))$}.

Thus, to finish the proof that $\Shind{(\CC,\BE)}$ is exact, it suffices to show that \eqref{eqn:the-seq-that-should-lie-in-tildeX} lies in $\rwt{\CXd{\BE}}$. For this, it is in turn enough to verify that 
\begin{equation}\label{eqn:seq-in-CX-BE}
\begin{tikzcd}[column sep=3cm]
\alpha
	\arrow{r}{(\begin{psmallmatrix}\id{A} \\ 0\end{psmallmatrix}, \begin{psmallmatrix}\id{C} \\ 0\end{psmallmatrix})}
& \beta 
	\arrow{r}{(\begin{psmallmatrix}0 \amph \id{E}\end{psmallmatrix},\begin{psmallmatrix}0 \amph \id{G}\end{psmallmatrix})}
&\gamma
\end{tikzcd}
\end{equation}
lies in $\CXd{\BE}$. By the arguments above, we already know that \eqref{eqn:seq-in-CX-BE} is a sequence of morphisms in $\catext{\BE}{\CC}$. As its underlying sequences are split exact in $\CC$, we have that \eqref{eqn:seq-in-CX-BE} lies in $\CXd{\BE}$.

We now show that $\Tsadid{(\CC,\BE)}$ is exact. Let 
\begin{equation}\label{eqn:Shin-inverse-exact-original-seq}
\begin{tikzcd}
(\alpha, (\ed{A},\ed{C}))
	\arrow{r}
& (\beta, (\ed{B},\ed{D}))
	\arrow{r}
& (\gamma, (\ed{E},\ed{G}))
\end{tikzcd}
\end{equation}
be a conflation in $\rwt{\CXd{\BE}}$. Consequently, we have that \eqref{eqn:Shin-inverse-exact-original-seq} is a direct summand of a sequence 
\begin{equation}\label{eqn:bigger-sequence-in-Shin-inverse}
\begin{tikzcd}
(\alpha', (\id{A'},\id{C'}))
	\arrow{r}
& (\beta', (\id{B'},\id{D'}))
	\arrow{r}
& (\gamma', (\id{E'},\id{G'})),
\end{tikzcd}
\end{equation}
which is the image under $\SId{\catext{\BE}{\CC}}$ of a kernel-cokernel pair 
\begin{equation}
\label{eqn:source-of-bigger-sequence-in-Shin-inverse}
\begin{tikzcd}[column sep=2cm]
\alpha'
	\arrow{r}{(a',c')}
& \beta'
	\arrow{r}{(b',d')}
& \gamma'
\end{tikzcd}
\end{equation}
in $\CXd{\BE}$. Apply $\Tsadid{(\CC,\BE)}$ to \eqref{eqn:bigger-sequence-in-Shin-inverse} to obtain
\begin{equation}\label{eqn:Shin-inverse-applied-to-bigger-sequence}
\begin{tikzcd}
(\id{A'},\alpha',\id{C'})
	\arrow{r}
& (\id{B'},\beta',\id{D'})
	\arrow{r}
& (\id{E'},\gamma',\id{G'}).
\end{tikzcd}
\end{equation}
We claim that \eqref{eqn:Shin-inverse-applied-to-bigger-sequence} lies in $\CXd{\rwt{\BE}}$. Since \eqref{eqn:source-of-bigger-sequence-in-Shin-inverse} belongs to $\CXd{\BE}$, its underlying sequences are split exact in $\CC$. As $\SId{\CC} \colon \CC \to \rwt{\CC}$ is an additive functor, the sequences 
\begin{align*}
		\begin{tikzcd}[column sep=3cm,ampersand replacement=\&]
		(A', \id{A'})
			\arrow{r}{(\id{B'},a',\id{A'})}
		\&(B', \id{B'})
			\arrow{r}{(\id{E'},b',\id{B'})}
		\&(E', \id{E'}),
		\end{tikzcd}\\
		\begin{tikzcd}[column sep=3cm,ampersand replacement=\&]
		(C', \id{C'})
			\arrow{r}{(\id{D'},c',\id{C'})}
		\&(D', \id{D'})
			\arrow{r}{(\id{G'},d',\id{D'})}
		\&(G', \id{G'})\phantom{,}
		\end{tikzcd}
\end{align*}
are thus split exact in $\rwt{\CC}$, and so \eqref{eqn:Shin-inverse-applied-to-bigger-sequence} lies in $\CXd{\rwt{\BE}}$. 

Since \eqref{eqn:Shin-inverse-exact-original-seq} is a direct summand of \eqref{eqn:bigger-sequence-in-Shin-inverse}, we know that $\Tsadid{(\CC,\BE)}\eqref{eqn:Shin-inverse-exact-original-seq}$ is a direct summand of \mbox{$\Tsadid{(\CC,\BE)}\eqref{eqn:bigger-sequence-in-Shin-inverse}=\eqref{eqn:Shin-inverse-applied-to-bigger-sequence}\inn\CXd{\rwt{\BE}}$}. Thus, by \cite[Cor.~2.18]{Buhler-exact-categories}, we deduce that $\Tsadid{(\CC,\BE)}\eqref{eqn:Shin-inverse-exact-original-seq}$ belongs to $\CXd{\rwt{\BE}}$, and hence $\Tsadid{(\CC,\BE)}$ is an exact functor. 
\end{proof}

We finish this section by demonstrating the use of \cref{thm:cat-of-extensions-of-IC-isomorphic-to-IC-of-cat-of-extensions} in a concrete example.

\begin{example}
\label{example-matrices}
Let $R$ be a unital ring. 
For integers $m,n\geq0$, let $\Mat{m,n}(R)$ denote the set of $m\times n$ matrices which, when $m,n>0$, have entries in $R$. Note that $\Mat{m,0}(R)$ consists of a single empty column vector of length $m$. Likewise, the set $\Mat{0,n}(R)$ consists of an empty length $n$ row vector. 
By declaring the image of \mbox{$\Mat{l,0}(R)\times \Mat{0,n}(R)\to \Mat{l,n}(R)$} to be the zero matrix, there is a function \mbox{$\Mat{l,m}(R)\times \Mat{m,n}(R)\to \Mat{l,n}(R)$} defined for any integers $l,m,n\geq 0$ that extends matrix multiplication. When $l=0$, the codomain of this function contains a unique element as noted above. Therefore, in this case, this function has the effect of changing the length of the empty length $m$ row vector to length $n$. Similarly, if $n=0$, then the corresponding function changes the length $m$ column vector to length $l$. 

Let $\CM$ be the category of \emph{rectangular matrices over} $R$, defined as follows. Objects of $\CM$ are rectangular matrices $X\inn \Mat{m,n}(R)$ for $m,n\geq0$, and a morphism from $X\inn \Mat{m,n}(R)$ to $Y\inn \Mat{p,q}(R)$ is defined by a pair of matrices $(A,B)\inn \Mat{q,n}(R)\times \Mat{p,m}(R)$ such that  $BX=YA$. Composition is defined by component-wise matrix multiplication. The identity of an object $X\inn \Mat{m,n}(R)$ is the pair $(\tensor*[]{I}{_{n}},\tensor*[]{I}{_{m}})$ of identity matrices. 

The category $\CM$ is preadditive, where the addition of morphisms is given component-wise. Furthermore, $\CM$ is in fact additive, where the direct sum $X\oplus Y$ of $X\inn \Mat{m,n}(R)$ and $Y\inn \Mat{p,q}(R)$ is given by the block matrix in $\Mat{m+p,n+q}(R)$ formed by taking $X$ and $Y$ in the diagonal blocks and $0$ elsewhere. 
That is,
\[
X\oplus Y=
\begin{pNiceArray}{c|c}
 \, X & \mathbf{0}\,\\
  \hline
  \, \mathbf{0} & Y\,
\end{pNiceArray}.
\]
If $m=0$, then $X$ is an empty row  and $X\oplus Y$ is found by inserting $n$ columns, all with entries equal to $0$,  to the left of $Y$. Likewise: if $n=0$, one inserts $0$-rows above $Y$; if $p=0$, one inserts $0$-columns to the right of $X$; and if $q=0$, one inserts $0$-rows below $X$. The zero object of $\CM$ is the unique element of $\Mat{0,0}(R)$. 

For integers $n,m\geq 0$, recall that $R^{m}\cong R^{n}$ as left (or, in fact, right) $R$-modules if and only if $AB=I_{m}$ and $BA=I_{n}$ for some $(A,B)\in \Mat{m,n}(R)\times \Mat{n,m}(R)$. 
In this case, let us write $m\sim n$. The ring $R$ is said to have the \emph{invariant basis number} (IBN) property provided that $m\sim n$ implies $m=n$; see \cite[p.\ 60]{Rotman}. Any one-sided noetherian ring has the IBN property \cite[Thm.~3.24]{Rotman}, as does any commutative ring \cite[Prop.~2.37]{Rotman}. The endomorphism ring of a  vector space of countably infinite  dimension does not have the IBN property \cite[Exa.~2.36]{Rotman}. 

For what remains of \cref{example-matrices}, we assume that $R$ has the IBN property. 
We now use \cref{thm:cat-of-extensions-of-IC-isomorphic-to-IC-of-cat-of-extensions} to compute the idempotent completion $\rwt{\CM}$. Consider first the category $\CC$ of finitely generated free $R$-modules, and let $\BE\colon \CCu{\op}\times \CC\to \Ab$ be the biadditive functor given by the $\Hom$-bifunctor $\CC(-,-)$. Using that any finitely generated free $R$-module is isomorphic to $\tensor*[]{R}{^{n}}$ for some $n\geq 0$, it is straightforward to check that $\CM$ is equivalent to the category $\catext{\BE}{\CC}$ of extensions, because $R$ has the IBN property. 

Applying \cref{thm:cat-of-extensions-of-IC-isomorphic-to-IC-of-cat-of-extensions} now gives $\rwt{\CM}\simeq \catext{\rwt{\BE}}{\rwt{\CC}}$. Observing that $\rwt{\BE} = \rwt{\CC}(-,-)$, it follows from \cite[Exa.~3.3]{Bennett-TennenhausHauglandSandoyShah-the-category-of-extensions-and-a-characterisation-of-n-exangulated-functors} that $\catext{\rwt{\BE}}{\rwt{\CC}}$ is equivalent to the \textit{arrow category} $\rwt{\CC}^{\rightarrow}$ of $\rwt{\CC}$, i.e.\ the category whose objects and morphisms are given by morphisms and commutative squares in $\rwt{\CC}$, respectively. For simplicity, assume from here that $R$ is commutative. It is well-known that the idempotent completion of $\CC$ is the category $\CP$ of finitely generated projective $R$-modules; see e.g.\ \cite[Exa.~2]{Borceux-Dejean-Cauchy-completion-in-category-theory}. We can thus conclude that $\rwt{\CM}$ is equivalent to the arrow category $\CP^{\rightarrow}$.
\end{example}


\section{\texorpdfstring{$2$}{2}-categorical compatibility}
\label{sec:2-categorical-perspective}

The aim of this section is to prove \cref{thm-C} in \cref{sec:introduction}, which asserts that the constructions and results we have exhibited so far are compatible in a $2$-categorical framework. 
We start by recalling the definition of the $2$-functor $\updave$ from \cite{Bennett-TennenhausHauglandSandoyShah-the-category-of-extensions-and-a-characterisation-of-n-exangulated-functors}.

Given an $n$-exangulated functor $(\SF,\Gamma)\colon(\CC,\BE,\fs)\to(\CD,\BF,\ft)$, it follows from \cite[Prop.~3.11, Thm.~3.17]{Bennett-TennenhausHauglandSandoyShah-the-category-of-extensions-and-a-characterisation-of-n-exangulated-functors} that there is a corresponding exact functor
\[
\SEd{(\SF,\Gamma)}\colon (\catext{\BE}{\CC},\CXd{\BE})\to (\catext{\BF}{\CD},\CXd{\BF}).
\] 
This functor is defined by $\SEd{(\SF,\Gamma)}(\alpha) = \Gamma(\alpha)$ on objects and by \mbox{$\SEd{(\SF,\Gamma)}(a,c) = (\SF a,\SF c)$} on morphisms. In addition, given an $n$-exangulated natural transformation \mbox{$\Beth \colon (\SF,\Gamma) \Rightarrow (\SG,\Lambda)$} of $n$-exangulated functors \mbox{$(\SF,\Gamma), (\SG,\Lambda) \colon(\CC,\BE,\fs)\to(\CD,\BF,\ft)$}, one can define a natural transformation \mbox{$\lan \Beth \ran \colon \SEd{(\SF,\Gamma)} \Rightarrow \SEd{(\SG,\Lambda)}$} given by \mbox{$\tensor*[]{\lan \Beth \ran}{_{\alpha}}  = (\Bethd{A},\Bethd{C})$} for $\alpha\inn\BE(C,A)$; see \cite[Thm.~4.19]{Bennett-TennenhausHauglandSandoyShah-the-category-of-extensions-and-a-characterisation-of-n-exangulated-functors}.

\begin{defn}
\label{def:updave}
(See \cite[Def.~4.20]{Bennett-TennenhausHauglandSandoyShah-the-category-of-extensions-and-a-characterisation-of-n-exangulated-functors}.)
Let $\updave = (\davecell{0},\davecell{1},\davecell{2}) \colon \Exang{n} \to \Exactcat$ be defined by the assignments $\davecell{i}\colon \Exangcell{n}{i} \to \Exactcatcell{i}$, where:
\[
\davecell{0}  (\CC,\BE,\fs)  
    \deff (\catext{\BE}{\CC},\CXd{\BE}),
\hspace{1cm}
\davecell{1}  (\SF,\Gamma)  
    \deff \SEd{(\SF,\Gamma)},
\hspace{1cm}
\davecell{2} ( \Beth ) 
    \deff\lan\Beth\ran.
\]
\end{defn}

It was shown in \cite[Thm.~4.22]{Bennett-TennenhausHauglandSandoyShah-the-category-of-extensions-and-a-characterisation-of-n-exangulated-functors} that $\updave$ defines a functor $\Exang{n} \to \Exactcat$ that satisfies the properties of a $2$-functor, and thus restricts to a genuine $2$-functor \mbox{$\exang{n} \to \exactcat$}. \cref{prop-A} allow us to restrict $\updave$ to idempotent complete categories. By abuse of notation, we write $\rwt{\updave}$ for this restriction, where it should be noted that $\rwt{\updave}$ is not the completion of the functor $\updave$ in the sense of \cref{sec:idempotent-completion-exact-cats} (see \eqref{eqn:IC-of-functor}). We have 
\[
\rwt{\updave} =  (\rwtdavecell{0},\rwtdavecell{1},\rwtdavecell{2})
	\colon \ICExang{n} \to \ICExactcat,
\]
where the assignment $\rwt{\updave}_i$ is defined as the restriction of $\davecell{i}$ for $i \inn \{0,1,2\}$ and satisfies the same properties. Again, we obtain a $2$-functor $\rwt{\updave} \colon \ICexang{n} \to \ICexactcat$ when restricting $0$-cells to small categories.

We frequently use that any $n$-exangulated category $(\CC,\BE,\fs)$ gives rise to the pair $(\CC,\BE)$ of an additive category equipped with a biadditive functor by forgetting the realisation $\fs$. In particular, for each \mbox{$(\CC,\BE,\fs)\inn\Exangcell{n}{0}$}, there is an exact equivalence 
\[
\Shind{(\CC,\BE)} \colon 
        (\catext{\rwt{\BE}}{\rwt{\CC}}, \CXd{\rwt{\BE}})
            \to 
	(\rwt{\catext{\BE}{\CC}}, \rwt{\CXd{\BE}}),
\]
which was defined in the proof of \textup{\cref{thm:cat-of-extensions-of-IC-isomorphic-to-IC-of-cat-of-extensions}}. Furthermore, recall that the functors $\exactcomplete$ and $\nexangcomplete$ were defined in Definitions~\ref{def:exactcomplete} and \ref{def:exangcomplete}, respectively. 

\begin{thm}
\label{thm:commutative-square} 
The collection $\Shin$ of exact equivalences $\Shind{(\CC,\BE)}$ for $(\CC,\BE,\fs)\inn\Exangcell{n}{0}$ defines a natural transformation $\rwt{\updave} \nexangcomplete \Rightarrow \exactcomplete \updave$ as indicated in the diagram 
\[
\begin{tikzcd}[column sep=1cm]
\Exang{n} 
	\arrow{d}[swap]{\nexangcomplete} 
	\arrow{r}{\updave} 
& \Exactcat 
	\arrow{d}{\exactcomplete}\\
\ICExang{n} 
	\arrow{r}[swap]{\rwt{\updave}} 
	\arrow[Rightarrow, 
			shorten <= 18pt,
            shorten >= 20pt,yshift={-5pt}, xshift={2pt}]{ur}{\Shin}
& \ICExactcat.
\end{tikzcd}
\]
\end{thm}

\begin{proof}
In order to demonstrate the naturality of $\Shin$, we must show that 
\[
\begin{tikzcd}[column sep=2cm]
(\catext{\rwt{\BE}}{\rwt{\CC}},\CXd{\rwt{\BE}}) 
	\arrow{r}{\Shind{(\CC,\BE)}}
	\arrow{d}[swap]{\SEd{(\rwt{\SF},\rwt{\Gamma})}}
& (\rwt{\catext{\BE}{\CC}},\rwt{\CXd{\BE}})  
	\arrow{d}{\rwt{\SEd{(\SF,\Gamma)}}}\\
(\catext{\rwt{\BF}}{\rwt{\CD}},\CXd{\rwt{\BF}}) 
	\arrow{r}{\Shind{(\CD,\BF)}}
& (\rwt{\catext{\BF}{\CD}},\rwt{\CXd{\BF}}) 
\end{tikzcd}
\]
commutes in $\ICExactcat$ for any $n$-exangulated functor \mbox{$(\SF,\Gamma)\colon (\CC,\BE,\fs) \to (\CD,\BF,\ft)$}. That is, we need to show that $\rwt{\SEd{(\SF,\Gamma)}}\Shind{(\CC,\BE)}$ and $\Shind{(\CD,\BF)}\SEd{(\rwt{\SF},\rwt{\Gamma})}$ are equal as functors \mbox{$\catext{\rwt{\BE}}{\rwt{\CC}} \to \rwt{\catext{\BF}{\CD}}$}.

To this end, let $(\ed{A},\alpha,\ed{C})$ be an object in $\catext{\rwt{\BE}}{\rwt{\CC}}$. 
On the one hand, we have that
\[
\rwt{\SEd{(\SF,\Gamma)}} \Shind{(\CC,\BE)}(\ed{A},\alpha,\ed{C}) 
	= \rwt{\SEd{(\SF,\Gamma)}} (\alpha,(\ed{A},\ed{C})) 
	= (\Gamma(\alpha),(\SF \ed{A},\SF \ed{C})), 
\]
while on the other hand
\[
\Shind{(\CD,\BF)}  \SEd{(\rwt{\SF},\rwt{\Gamma})} (\ed{A},\alpha, \ed{C}) 
	= \Shind{(\CD,\BF)} (\SF \ed{A},\Gamma(\alpha), \SF \ed{C}) 
	= (\Gamma(\alpha),(\SF \ed{A},\SF \ed{C})).
\]
Hence, the functors $\rwt{\SEd{(\SF,\Gamma)}}\Shind{(\CC,\BE)}$ and $\Shind{(\CD,\BF)}\SEd{(\rwt{\SF},\rwt{\Gamma})}$ agree on objects. Consider next a morphism \mbox{$((\ed{B},a,\ed{A}),(\ed{D},c,\ed{C}))\colon (\ed{A},\alpha, \ed{C}) \to (\ed{B},\beta, \ed{D})$ in $\catext{\rwt{\BE}}{\rwt{\CC}}$}. 
We have 
\begin{align*}
\rwt{\SEd{(\SF,\Gamma)}} \Shind{(\CC,\BE)}((\ed{B},a,\ed{A}),(\ed{D},c,\ed{C}))
	&= \rwt{\SEd{(\SF,\Gamma)}} ((\ed{B},\ed{D}),(a,c),(\ed{A},\ed{C}))\\
	&= ((\SF \ed{B},\SF \ed{D}),(\SF a,\SF c),(\SF \ed{A},\SF \ed{C})) \\
	&= \Shind{(\CD,\BF)} ((\SF \ed{B},\SF a,\SF \ed{A}),(\SF \ed{D},\SF c,\SF \ed{C}))\\
	&= \Shind{(\CD,\BF)} \SEd{(\rwt{\SF},\rwt{\Gamma})}  ((\ed{B},a,\ed{A}),(\ed{D},c,\ed{C})).&\hspace{47pt}\qedhere
\end{align*}
\end{proof}

The next result says that $\Shin$ satisfies the defining property of a \emph{$2$-natural transformation} between $2$-functors; see \cite[Prop.~4.2.11]{JohnsonYau-2-dimensional-categories}. 

\begin{prop}
\label{prop:IC-2-natural}
Let $\Beth \colon (\SF,\Gamma) \Rightarrow (\SG,\Lambda)$ be an $n$-exangulated natural transformation between $n$-exangulated functors \mbox{$(\SF,\Gamma), (\SG,\Lambda) \colon(\CC,\BE,\fs)\to(\CD,\BF,\ft)$}. Then the square 
\[
\begin{tikzcd}
(\exactcomplete\updave)(\SF,\Gamma)\circ\Shind{(\CC,\BE)}
	\arrow[equal]{r}
	\arrow[Rightarrow]{d}[swap]{(\exactcomplete\updave)(\Beth)\horicomp \id{\Shind{(\CC,\BE)}}}
& \Shind{(\CD,\BF)} \circ (\rwt{\updave}\nexangcomplete)(\SF,\Gamma) 
	\arrow[Rightarrow]{d}[xshift=2pt, yshift=0pt]{\id{\Shind{(\CD,\BF)}}\horicomp(\rwt{\updave}\nexangcomplete)(\Beth)}\\
(\exactcomplete\updave)(\SG,\Lambda)\circ\Shind{(\CC,\BE)}
	\arrow[equal]{r}
& \Shind{(\CD,\BF)} \circ(\rwt{\updave}\nexangcomplete)(\SG,\Lambda) 
\end{tikzcd}
\]
commutes in $\ICExactcat(\catext{\rwt{\BE}}{\rwt{\CC}} , \rwt{\catext{\BF}{\CD}})$.
\end{prop}

\begin{proof}
Note first that we have the horizontal equalities by \cref{thm:commutative-square}. Consider an arbitrary object \mbox{$(\ed{A},\alpha,\ed{C})\inn\rwt{\BE}((C,\ed{C}),(A,\ed{A}))$} in \mbox{$\catext{\rwt{\BE}}{\rwt{\CC}} = (\rwt{\updave}\nexangcomplete)(\CC,\BE,\fs)$}. 
On the one hand,  
\begin{align*}
((\exactcomplete\updave)(\Beth) \horicomp \id{\Shind{(\CC,\BE)}}\braced{(\ed{A},\alpha,\ed{C})}
	&= (\exactcomplete\updave)(\Beth\braced{\Shind{(\CC,\BE)}(\ed{A},\alpha,\ed{C}) } \circ (\exactcomplete\updave)(\SF,\Gamma)(\id{\Shind{(\CC,\BE)}}\braced{(\ed{A},\alpha,\ed{C})}\\
    &= (\exactcomplete\updave)(\Beth\braced{(\alpha,(\ed{A},\ed{C}))} \circ \rwt{\SEd{(\SF,\Gamma)}}(\id{(\alpha,(\ed{A},\ed{C}))}) \\
    &= (\exactcomplete\updave)(\Beth\braced{(\alpha,(\ed{A},\ed{C}))} \\
	&= (\exactcomplete \updave (\Beth)\braced{(\alpha,(\ed{A},\ed{C}))} \\
	&= \tensor*[]{\rwt{\lan\Beth\ran}}{_{(\alpha,(\ed{A},\ed{C}))}} \\
	&= (\SEd{(\SG,\Lambda)} (\ed{A},\ed{C}),\SEd{(\SG,\Lambda)} (\ed{A},\ed{C})\tensor*[]{\lan\Beth\ran}{_{\alpha}}\SEd{(\SF,\Gamma)}(\ed{A},\ed{C}),\SEd{(\SF,\Gamma)} (\ed{A},\ed{C}))\\
	&= ((\SG \ed{A},\SG \ed{C}), 
		((\SG \ed{A})\Bethd{A}\SF \ed{A}, (\SG \ed{C})\Bethd{C}\SF \ed{C}),
		(\SF \ed{A},\SF \ed{C})). 
\end{align*}
On the other hand, we have
\begin{align*}
(\id{\Shind{(\CD,\BF)}} \hspace{-1.5mm} \horicomp \hspace{-0.75mm}(\rwt{\updave}\nexangcomplete)(\Beth)\braced{(\ed{A},\alpha,\ed{C})}
	&= (\id{\Shind{(\CD,\BF)}}\braced{(\rwt{\updave}\nexangcomplete)(\SG,\Lambda)(\ed{A},\alpha,\ed{C})}
	\circ \Shind{(\CD,\BF)}((\rwt{\updave}\nexangcomplete)(\Beth)\braced{(\ed{A},\alpha,\ed{C})}\\
        &= \Shind{(\CD,\BF)}\tensor*[]{\lan\rwt{\Beth}\ran}{_{(\ed{A},\alpha,\ed{C})}}\\
	&= \Shind{(\CD,\BF)}
			( (\SG \ed{A},\hspace{-0.75mm} (\SG \ed{A})\hspace{-0.75mm}\Bethd{A}\hspace{-0.75mm}\SF \ed{A}, \SF \ed{A}) ,\hspace{-0.75mm} (\SG \ed{C},\hspace{-0.75mm} (\SG \ed{C})\hspace{-0.75mm}\Bethd{C}\hspace{-0.75mm}\SF \ed{C}, \SF \ed{C}) )\\
	&= ((\SG \ed{A},\SG \ed{C}), 
		((\SG \ed{A})\Bethd{A}\SF \ed{A},(\SG \ed{C})\Bethd{C}\SF \ed{C}),
		(\SF \ed{A},\SF \ed{C})),
\end{align*}
which finishes the proof. 
\end{proof}

Restricting to small categories, \cref{thm:commutative-square} and \cref{prop:IC-2-natural} yield the following corollary, demonstrating that idempotent completions and extension categories are compatible constructions in a $2$-category-theoretic sense.

\begin{cor}
\label{cor:commutativity-of-2-square}
There is a $2$-natural transformation 
$
\Shin 
	\colon 
	\rwt{\updave} \nexangcomplete
		\Rightarrow
	\exactcomplete\updave
$
of $2$-functors from $\exang{n}$ to $\ICexactcat$ consisting of exact equivalences.  
\end{cor}


\section{The weak idempotent completion}
\label{sec:weak-idempotent-completion}

The aim of this section is to relate our main results to the construction of weak idempotent completions. Note that the definitions and results in \cref{sec:weak-idempotent-completion} rely on concepts and notation which should be recalled from previous sections. Many of the proofs in the weakly idempotent complete case are straightforward modifications of those for the idempotent completion. However, we remark that our proof of the key result \cref{prop:idempotents-split-category-of-extensions-WIC} differs significantly from the proof of \cref{prop-A} and relies on a result from \cite{Bennett-TennenhausHauglandSandoyShah-the-category-of-extensions-and-a-characterisation-of-n-exangulated-functors}.

Recall that an additive category $\CC$ is said to be \emph{weakly idempotent complete} if every retraction has a kernel or, equivalently, if every section has a cokernel; see \cite[Lem.~7.1, Def.~7.2]{Buhler-exact-categories}. Every idempotent complete category is weakly idempotent complete; see e.g.\ \mbox{\cite[Prop.~6.5.4]{Borceux-handbook-1}} and \cite[Lem.~A.6.2]{ThomasonTrobaugh-higher-algebraic-K-theory-of-schemes-and-of-derived-categories}. For more detail on weak idempotent completions, see e.g.\ \cite[Sec.~2.2]{KlapprothMsapatoShah-Idempotent-completions-of-n-exangulated-categories}. 

\begin{defn}
\label{def:weak-idempotent-completion}
(See \cite[Def.~2.10]{KlapprothMsapatoShah-Idempotent-completions-of-n-exangulated-categories}.) 
The \emph{weak idempotent completion $\rwh{\CC}$ of $\CC$} is the full subcategory of $\rwt{\CC}$ that consists of all objects $(X,e)$ for which $\id{X}-e$ splits in $\CC$. 
\end{defn}

Note that $\rwh{\CC}$ is an additive subcategory of $\rwt{\CC}$ and that there is a canonical additive inclusion functor \mbox{$\SKd{\CC}\colon \CC \to \rwh{\CC}$} defined by $\SKd{\CC}(X) \deff (X,\id{X})$ for $X\inn\CC$ and \mbox{$\SKd{\CC}(f) \deff (\id{Y},f,\id{X})$} for \mbox{$f\inn\CC(X,Y)$}. This functor is $2$-universal among additive functors from $\CC$ to weakly idempotent complete categories; see \cite[Prop.~2.13]{KlapprothMsapatoShah-Idempotent-completions-of-n-exangulated-categories}. The functor \mbox{$\SId{\CC} \colon \CC \to \rwt{\CC}$} factors through $\SKd{\CC}$ via the canonical inclusion functor \mbox{$\SLd{\CC}\colon \rwh{\CC} \to \rwt{\CC}$}, which is the identity on objects and morphisms. In other words, there is a commutative diagram
\begin{equation}
\label{eqn:WIC3}
\begin{tikzcd}
\CC 
    \arrow{rr}{\SId{\CC}}
    \arrow{dr}[swap]{\SKd{\CC}}
&{}
& \rwt{\CC} \\
& \rwh{\CC}
    \arrow{ur}[swap]{\SLd{\CC}}
&{}
\end{tikzcd}
\end{equation}
of additive categories and functors.

Suppose that $(\CC,\CX)$ is an exact category. One defines an exact structure $\rwh{\CX}$ on $\rwh{\CC}$ as follows. An object of $\rwhCCu{\rightarrow\rightarrow}$ is in $\rwh{\CX}$ if it is a direct summand of an object in the image of $\CX$ under the functor \mbox{$\CCu{\rightarrow\rightarrow}\to \rwhCCu{\rightarrow\rightarrow}$} induced by \mbox{$\SKd{\CC} \colon \CC \to \rwh{\CC}$}. In particular, a kernel-cokernel pair lies in $\rwh{\CX}$ if and only if it is a kernel-cokernel pair in $\rwt{\CX}$ in which all three terms lie in $\rwh{\CC}$. 

\cref{prop:WIC-is-exact-subcategory} below shows that $(\rwh{\CC},\rwh{\CX})$ is a \emph{fully exact subcategory} of $(\rwt{\CC},\rwt{\CX})$. This means that $\rwh{\CC}$ is extension-closed in $(\rwt{\CC},\rwt{\CX})$ and that $\rwh{\CX}$ coincides with the inherited exact structure; see \cite[Lem.~10.20, Def.~10.21]{Buhler-exact-categories}.

\begin{prop}
\label{prop:WIC-is-exact-subcategory}
The pair $(\rwh{\CC},\rwh{\CX})$ is a fully exact subcategory of $(\rwt{\CC},\rwt{\CX})$. The inclusion \mbox{$\SKd{\CC} \colon (\CC,\CX) \to (\rwh{\CC},\rwh{\CX})$} is a fully faithful exact functor that reflects exactness.
\end{prop}

\begin{proof}
Suppose that 
$
\begin{tikzcd}[column sep=1.7cm]
(A,\ed{A})
    \arrow{r}{(\ed{B}, a, \ed{A})}
& (B,\ed{B})
    \arrow{r}{(\ed{C}, b, \ed{B})}
& (C,\ed{C})
\end{tikzcd}
$
is a conflation in $(\rwt{\CC},\rwt{\CX})$ with \mbox{$(A,\ed{A}),(C,\ed{C})\inn\rwh{\CC}$}. By \cite[Prop.~5.1]{KlapprothMsapatoShah-Idempotent-completions-of-n-exangulated-categories}, there is an object $(D,\ed{D})\inn\rwh{\CC}$ and an isomorphism 
\begin{equation}
\label{eqn:WIC2}
\begin{tikzcd}[column sep=2cm, row sep=1.2cm]
(A,\ed{A})
    \arrow{r}{(\ed{B}, a, \ed{A})}
    \arrow[equals]{d}
& (B,\ed{B})
    \arrow{r}{(\ed{C}, b, \ed{B})}
    \arrow{d}{(\ed{D}, r, \ed{B})}
& (C,\ed{C})
    \arrow[equals]{d} \\
(A,\ed{A})
    \arrow{r}{(\ed{D}, c, \ed{A})}
& (D,\ed{D})
    \arrow{r}{(\ed{C}, d, \ed{D})}
& (C,\ed{C})
\end{tikzcd}
\end{equation}
in the category $\tensor*[]{\mathbf{K}}{^{3}_{(\rwt{\CC};(A,\ed{A}),(C,\ed{C}))}}$ defined in \cite[Def.~2.17]{HerschendLiuNakaoka-n-exangulated-categories-I-definitions-and-fundamental-properties}. By \cite[Lem.~4.1]{HerschendLiuNakaoka-n-exangulated-categories-I-definitions-and-fundamental-properties}, the morphism $(\ed{D}, r, \ed{B})$ is an isomorphism in $\rwt{\CC}$, so $\rwh{\CC}$ is extension-closed. Since a sequence lies in $\rwh{\CX}$ if and only if it lies in $\rwt{\CX}$ and has all terms in $\rwh{\CC}$, the inherited exact structure (see \cite[Lem.~10.20]{Buhler-exact-categories}) coincides with $\rwh{\CX}$. This shows that $(\rwh{\CC},\rwh{\CX})$ is a fully exact subcategory of $(\rwt{\CC},\rwt{\CX})$.

As a consequence, we observe that $\SLd{\CC} \colon (\rwh{\CC},\rwh{\CX}) \to (\rwt{\CC},\rwt{\CX})$ from \eqref{eqn:WIC3} is an exact functor. One sees directly that $\SKd{\CC} \colon (\CC,\CX) \to (\rwh{\CC},\rwh{\CX})$ is fully faithful and exact. That it reflects exactness follows from \mbox{$\SId{\CC}\colon (\CC,\CX) \to (\rwt{\CC},\rwt{\CX})$} reflecting exactness (see \cref{prop:inclusion-functor-into-Karoubi-envelope-is-ff-exact}) and the commutative diagram \eqref{eqn:WIC3}.
\end{proof}

Suppose that $(\CD,\CY)$ is also an exact category. Let $\SF,\SG\colon (\CC,\CX)\to (\CD,\CY)$ be exact functors and consider a natural transformation $\Beth\colon\SF \Rightarrow \SG$. The completions $\rwt{\SF}$ and $\rwt{\SG}$ restrict to give exact functors \mbox{$\rwh{\SF},\rwh{\SG}\colon (\rwh{\CC},\rwh{\CX})\to (\rwh{\CD},\rwh{\CY})$}. Moreover, $\rwt{\Beth}$ restricts to a natural transformation 
\[
\rwh{\Beth} 
    \deff \Setd{ \rwtBethd{(X,e)} }{(X,e)\inn\rwh{\CC}}
    \colon \rwh{\SF}\Rightarrow\rwh{\SG}.
\]
 
We write $\WICExactcat$ and $\WICexactcat$ for the restrictions to weakly idempotent complete $0$-cells in $\Exactcat$ and $\exactcat$, respectively, and note that $\WICexactcat$ is a $2$-category. 

\begin{defn}
\label{def:exactWIC}
Let $\exactWIC = (\exactWICcell{0},\exactWICcell{1},\exactWICcell{2}) \colon \Exactcat \to \WICExactcat$ be defined by the assignments
$\exactWICcell{i}\colon \Exactcatcell{i} \to \WICExactcatcell{i}$, 
where:
\[
\exactWICcell{0}  (\CC,\CX)  
    \deff (\rwh{\CC},\rwh{\CX}),
\hspace{2cm}
\exactWICcell{1}  (\SF)  
    \deff \rwh{\SF},
\hspace{2cm}
\exactWICcell{2} ( \Beth ) 
    \deff \rwh{\Beth}.
\]
\end{defn}

These assignments are well-defined by the discussion above. As a consequence of \cref{thm:2-functor-of-exact-categories}, we see that $\exactWIC$ satisfies the properties of a $2$-functor, because $\rwh{\SF}$ and $\rwh{\Beth}$ are just restrictions of $\rwt{\SF}$ and $\rwt{\Beth}$, respectively. Thus, one deduces an analogue of \cref{thm:2-functor-of-exact-categories}. 

Throughout the rest of this section, suppose that $(\CC,\BE,\fs)$ is an $n$-exangulated category. The next result is an analogue of \cref{prop-A}, but interestingly the proof is very different. 

\begin{prop}
\label{prop:idempotents-split-category-of-extensions-WIC} 
If $\CC$ is weakly idempotent complete, then $\catext{\BE}{\CC}$ is also weakly idempotent complete.
\end{prop}

\begin{proof}
Let $(a,c)\colon \alpha \to \beta$ be a morphism in $\catext{\BE}{\CC}$ for $\alpha\inn\BE(C,A)$ and \mbox{$\beta\inn\BE(D,B)$}. Suppose that this morphism is a section. Thus, there is a retraction \mbox{$(r,s)\colon \beta \to \alpha$} in $\catext{\BE}{\CC}$ satisfying \mbox{$(ra,sc) = (r,s)\circ (a,c) = \id{\alpha} = (\id{A},\id{C})$}. In particular, this implies that $a$ and $c$ are sections in $\CC$. Since $\CC$ is weakly idempotent complete, these morphisms each admit a cokernel, which we denote by $b = \cok a$ and $d = \cok c$. We have that $(b,d)$ is a cokernel of $(a,c)$ in $\catext{\BE}{\CC}$ by \mbox{\cite[Lem.~3.1]{Bennett-TennenhausHauglandSandoyShah-the-category-of-extensions-and-a-characterisation-of-n-exangulated-functors}}, so $\catext{\BE}{\CC}$ is weakly idempotent complete.
\end{proof}

It was shown in \cite[Sec.~5]{KlapprothMsapatoShah-Idempotent-completions-of-n-exangulated-categories} that the weak idempotent completion $\rwh{\CC}$ of $\CC$ admits an \mbox{$n$-exangulated} structure because it is an \emph{extension-closed} subcategory of $(\rwt{\CC},\rwt{\BE},\rwt{\fs}\mspace{1mu})$ in the sense of \cite[Def.~4.1]{HerschendLiuNakaoka-n-exangulated-categories-II}. We denote the corresponding $n$-exangulated category by $(\rwh{\CC},\rwh{\BE},\rwh{\fs}\mspace{3mu})$. We now recall how $\rwh{\BE}$ and $\rwh{\fs}$ are defined.

\begin{defn}
\label{def:hat-BE-fs}
(See \cite[Def.~5.3]{KlapprothMsapatoShah-Idempotent-completions-of-n-exangulated-categories}.)
The biadditive functor $\rwh{\BE}\colon \rwhCCu{\op}\times\rwh{\CC}\to \Ab$ is the restriction of $\rwt{\BE}$ to $\rwhCCu{\op}\times\rwh{\CC}$. The exact realisation $\rwh{\fs}$ of $\rwh{\BE}$ is defined as follows. Suppose that \mbox{$(\ed{A},\alpha,\ed{C})\inn\rwh{\BE}((C,\ed{C}),(A,\ed{A}))$} is an $\rwh{\BE}$-extension. Then \mbox{$\rwt{\fs}(\ed{A},\alpha,\ed{C}) = [ \rwhXd{\bullet}]$} for some \mbox{$(n+2)$-term} complex $\rwhXd{\bullet}$ in $\rwh{\CC}$ by \cite[Prop.~5.1]{KlapprothMsapatoShah-Idempotent-completions-of-n-exangulated-categories}. 
Thus, it is declared that \mbox{$\rwh{\fs}(\ed{A},\alpha,\ed{C}) = [ \rwhXd{\bullet} ]$}.
\end{defn}

\cref{prop:idempotents-split-category-of-extensions-WIC} has the following immediate corollary. 

\begin{cor}
\label{cor:catext-C-hat-is-WIC}
$(\catext{\rwh{\BE}}{\rwh{\CC}}, \CXd{\rwh{\BE}})$ is a weakly idempotent complete exact category. 
\end{cor}
 
Recall from the proof of \cref{thm:cat-of-extensions-of-IC-isomorphic-to-IC-of-cat-of-extensions} that we have an exact equivalence 
\[
\Shind{(\CC,\BE)} \colon 
        (\catext{\rwt{\BE}}{\rwt{\CC}}, \CXd{\rwt{\BE}})
            \to 
	(\rwt{\catext{\BE}{\CC}}, \rwt{\CXd{\BE}})
\] 
with quasi-inverse $\Tsadid{(\CC,\BE)}$ for each $n$-exangulated category $(\CC,\BE,\fs)$ by forgetting the realisation $\fs$. It follows from \cref{prop:idempotents-split-category-of-extensions} that $\Shind{(\CC,\BE)}$ restricts to a functor 
\[
\Shindp{(\CC,\BE)}\colon
    (\catext{\rwh{\BE}}{\rwh{\CC}}, \CXd{\rwh{\BE}})
    \to (\rwh{\catext{\BE}{\CC}}, \rwhCXd{\BE}).
\]
To see this, let $(\ed{A},\alpha,\ed{C})\inn\rwh{\BE}((C,\ed{C}),(A,\ed{A}))$ and consider \mbox{$\Shind{(\CC,\BE)}(\ed{A},\alpha,\ed{C}) 
     = (\alpha,(\ed{A},\ed{C}))$}.
Note that \mbox{$\id{\alpha} - (\ed{A},\ed{C})\colon \alpha\to\alpha$} is a morphism of $\BE$-extensions since \mbox{$(\ed{A},\ed{C})\inn\Endd{\catext{\BE}{\CC}}(\alpha)$}. We must show that 
 \mbox{$\id{\alpha} - (\ed{A},\ed{C})
     = (\id{A}-\ed{A},\id{C}-\ed{C})$} 
splits in $\catext{\BE}{\CC}$. This follows from \cref{prop:idempotents-split-category-of-extensions}, as \mbox{$(A,\ed{A}),(C,\ed{C})\inn\rwh{\CC}$} means that $\id{A}-\ed{A}$ and $\id{C}-\ed{C}$ split in $\CC$.

A similar argument as above shows that $\Tsadid{(\CC,\BE)}$ restricts to a functor 
\[
\Tsadidp{(\CC,\BE)}\colon
    (\rwh{\catext{\BE}{\CC}}, \rwhCXd{\BE})
    \to 
    (\catext{\rwh{\BE}}{\rwh{\CC}}, \CXd{\rwh{\BE}}).
\]
Note that $\Shindp{(\CC,\BE)}$ and $\Tsadidp{(\CC,\BE)}$ are mutually quasi-inverse as they are restrictions of $\Shind{(\CC,\BE)}$ and $\Tsadid{(\CC,\BE)}$. Since it is straightforward to check that $\Shindp{(\CC,\BE)}$ and $\Tsadidp{(\CC,\BE)}$ preserve the exact structures, we have the following.

\begin{thm}
\label{prop:cat-of-extensions-of-WIC-isomorphic-to-WIC-of-cat-of-extensions}
There is an exact equivalence 
\mbox{$\Shindp{(\CC,\BE)}\colon (\catext{\rwh{\BE}}{\rwh{\CC}}, \CXd{\rwh{\BE}})
\to (\rwh{\catext{\BE}{\CC}}, \rwhCXd{\BE})$} 
given by the restriction of $\Shind{(\CC,\BE)}$. 
\end{thm}

Suppose that $(\SF,\Gamma)\colon (\CC,\BE,\fs)\to (\CD,\BF,\ft)$ is an $n$-exangulated functor. One can define a natural transformation \mbox{$\rwh{\Gamma}\colon \rwh{\BE}(-,-) \Rightarrow \rwh{\BF}(\rwh{\SF}-,\rwh{\SF}-)$} by setting 
\[
\rwhGammad{((C,\ed{C}),(A,\ed{A}))}(\ed{A}, \alpha, \ed{C}) 
    \deff  (\SF \ed{A}, \Gamma(\alpha), \SF \ed{C}).
\]
Notice that $\rwh{\Gamma}$ is just a restriction of $\rwt{\Gamma}$. We claim that the pair $(\rwh{\SF},\rwh{\Gamma})$ is an $n$-exangulated functor \mbox{$(\rwh{\CC},\rwh{\BE},\rwh{\fs}\mspace{3mu})\to(\rwh{\CD},\rwh{\BF},\rwh{\ft}\mspace{3mu})$}. To verify this, assume that \mbox{$\rwh{\fs}(\ed{A},\alpha,\ed{C}) = [\rwhXd{\bullet}]$}, which implies \mbox{$\rwt{\fs}(\ed{A},\alpha,\ed{C}) = [\rwhXd{\bullet}]$}. This yields 
\mbox{$\rwt{\ft}(\rwt{\Gamma}(\ed{A},\alpha,\ed{C})) 
    = [\rwtSFd{\com}(\rwhXd{\bullet})]$}, as $(\rwt{\SF},\rwt{\Gamma})$ is $n$-exangulated by \cref{prop:completion-of-n-ex-functor}. Since $\rwtSFd{\com}(\rwhXd{\bullet})$ is a complex in $\rwh{\CD}$, we obtain 
\[
\rwh{\ft}(\rwh{\Gamma}(\ed{A},\alpha,\ed{C})) 
    = [\rwtSFd{\com}(\rwhXd{\bullet})] 
    = [\rwhSFd{\com}(\rwhXd{\bullet})]
\] 
as required.

Let $\Beth \colon(\SF,\Gamma) \Rightarrow (\SG,\Lambda)$ be an $n$-exangulated natural transformation between $n$-exangulated functors \mbox{$(\SF,\Gamma),(\SG,\Lambda) \colon (\CC,\BE,\fs)\to (\CD,\BF,\ft)$}. Using that the completion \mbox{$\rwt{\Beth}\colon(\rwt{\SF},\rwt{\Gamma}) \Rightarrow (\rwt{\SG},\rwt{\Lambda})$} is an $n$-exangulated natural transformation by \cref{prop:completion-of-n-ex-naturaltr}, the same holds for the restriction \mbox{$\rwh{\Beth} \colon (\rwh{\SF},\rwh{\Gamma})\Rightarrow (\rwh{\SG},\rwh{\Lambda})$}. 

We write $\WICExang{n}$ for the collections obtained by only considering weakly idempotent complete $0$-cells in $\Exang{n}$. Based on the discussion above, we may thus define  
\[
\nexangWIC = (\nexangWICcell{0},\nexangWICcell{1},\nexangWICcell{2}) \colon \Exang{n} \to \WICExang{n}
\]
using assignments
\mbox{$\nexangWICcell{i}\colon \Exangcell{n}{i} \to \WICExangcell{n}{i}$}, 
where:
\[
\nexangWICcell{0}  (\CC,\BE,\fs)  
    \deff (\rwh{\CC},\rwh{\BE},\rwh{\fs}\mspace{3mu}),
\hspace{2cm}
\nexangWICcell{1}  (\SF,\Gamma)  
    \deff (\rwh{\SF},\rwh{\Gamma}),
\hspace{2cm}
\nexangWICcell{2} ( \Beth ) 
    \deff \rwh{\Beth}.
\]
It is straightforward to check that the analogue of \cref{thm:2-functor-of-n-exangulated-categories} holds for $\nexangWIC$.

As an application of \cref{prop:idempotents-split-category-of-extensions-WIC}, we can restrict $\updave$ to weakly idempotent complete categories. This restriction is denoted by
\[
\rwh{\updave} =  (\rwhdavecell{0},\rwhdavecell{1},\rwhdavecell{2}) \colon \WICExang{n} \to \WICExactcat,
\]
where $\rwhdavecell{i}$ is the restriction of $\davecell{i}$ for $i \inn \{0,1,2\}$ and satisfies the same properties. The proof of \cref{thm:commutative-square} yields the next theorem. Similarly, analogues of \cref{prop:IC-2-natural} and \cref{cor:commutativity-of-2-square} follow.

\begin{thm}
\label{thm:commutative-square-WIC} 
The collection $\Shin'$ of exact equivalences $\Shindp{(\CC,\BE)}$ for $(\CC,\BE,\fs)\inn\Exangcell{n}{0}$ defines a natural transformation $\rwh{\updave} \nexangWIC \Rightarrow \exactWIC \updave$. 
\end{thm}

By \cite[Prop.~4.36]{KlapprothMsapatoShah-Idempotent-completions-of-n-exangulated-categories}, there is an $n$-exangulated functor $(\SId{\CC},\Gammad{\CC})\colon (\CC,\BE,\fs) \to (\rwt{\CC},\rwt{\BE},\rwt{\fs}\mspace{1mu})$, where \mbox{$\Gammad{\CC}(\alpha) = (\id{A},\alpha,\id{C})$} for $\alpha\inn\BE(C,A)$. Similarly, it is shown in \cite[Thm.~5.5]{KlapprothMsapatoShah-Idempotent-completions-of-n-exangulated-categories} that \mbox{$(\SKd{\CC},\Deltad{\CC})\colon (\CC,\BE,\fs) \to (\rwh{\CC},\rwh{\BE},\rwh{\fs}\mspace{3mu})$} is $n$-exangulated, where $\Deltad{\CC}(\alpha) = (\id{A},\alpha,\id{C})$. Diagram \eqref{eqn:WIC3} can be augmented to a commutative diagram 
\begin{equation}
\label{eqn:WIC4}
\begin{tikzcd}
(\CC,\BE,\fs)
    \arrow{rr}{(\SId{\CC},\Gammad{\CC})}
    \arrow{dr}[swap]{(\SKd{\CC},\Deltad{\CC})}
&{}
& (\rwt{\CC},\rwt{\BE},\rwt{\fs}\mspace{1mu}) \\
& (\rwh{\CC},\rwh{\BE},\rwh{\fs}\mspace{3mu})
    \arrow{ur}[swap]{(\SLd{\CC},\Thetad{\CC})}
&{}
\end{tikzcd}
\end{equation}
in $\Exang{n}$, where $\Thetad{\CC}(\ed{B},\beta,\ed{D}) \deff (\ed{B},\beta,\ed{D})$ for $(\ed{B},\beta,\ed{D})\inn\rwh{\BE}((D,\ed{D}),(B,\ed{B}))$. Using the functor $\updave\colon\Exang{n} \to \Exactcat$, the diagram \eqref{eqn:WIC4} induces the commutative diagram 
\begin{equation*}
\label{eqn:WIC5}
\begin{tikzcd}
(\catext{\BE}{\CC},\CXd{\BE})
    \arrow{rr}{\SEd{(\SId{\CC},\Gammad{\CC})}}
    \arrow{dr}[swap]{\SEd{(\SKd{\CC},\Deltad{\CC})}}
&{}
& (\catext{\rwt{\BE}}{\rwt{\CC}},\CXd{\rwt{\BE}}) \\
& (\catext{\rwh{\BE}}{\rwh{\CC}},\CXd{\rwh{\BE}})
    \arrow{ur}[swap]{\SEd{(\SLd{\CC},\Thetad{\CC})}}
&{}
\end{tikzcd}
\end{equation*}
in $\Exactcat$.

Building on the work in \cref{sec:weak-idempotent-completion}, it is straightforward to check that we obtain \cref{thm:final_diagram} below. We note that there is a similar commutative diagram involving the exact equivalence $\Tsadid{(\CC,\BE)}$ and its restriction $\Tsadidp{(\CC,\BE)}$.

\begin{thm}\label{thm:final_diagram}
The diagram
\[\begin{tikzcd}[column sep=2cm]
    (\catext{\BE}{\CC}, \CXd{\BE})
        \arrow{r}{\SEd{(\SKd{\CC},\Deltad{\CC})}}
        \arrow[equals]{d}
    & (\catext{\rwh{\BE}}{\rwh{\CC}}, \CXd{\BE})
        \arrow{r}{\SEd{(\SLd{\CC},\Thetad{\CC})}}
        \arrow{d}{\Shindp{(\CC,\BE)}}[swap]{\simeq}
    & (\catext{\rwt{\BE}}{\rwt{\CC}}, \CXd{\BE})
        \arrow{d}{\Shind{(\CC,\BE)}}[swap]{\simeq} 
\\
    (\catext{\BE}{\CC}, \CXd{\BE})
        \arrow{r}{\SKd{\catext{\BE}{\CC}}}
    & (\rwh{\catext{\BE}{\CC}}, \CXd{\BE})
        \arrow{r}{\SLd{\catext{\BE}{\CC}}}
    & (\rwt{\catext{\BE}{\CC}}, \CXd{\BE})
\end{tikzcd}
\]
in $\Exactcat$ is commutative.
\end{thm}


{\begin{acknowledgements}
\setstretch{1}
The authors are grateful to Dixy Msapato for directing them to a result in their paper \cite{Msapato-the-karoubi-envelope-and-weak-idempotent-completion-of-an-extriangulated-category}, which led to \cref{prop:idempotents-split-category-of-extensions}. 

Parts of this work were carried out while the second author visited Aarhus University, and while the first and fourth author visited NTNU. The authors thank the project ``Pure Mathematics in Norway'' funded by the Trond Mohn Foundation for supporting these stays.

The first author is grateful to have been supported during part of this work by the Alexander von Humboldt Foundation in the framework of an Alexander von Humboldt Professorship endowed by the German Federal Ministry of Education. 
The third author is grateful to have been supported by Norwegian Research Council project 301375, ``Applications of reduction techniques and computations in representation theory''.
The fourth author is grateful to have been supported during part of this work by the Engineering and Physical Sciences Research Council (grant EP/P016014/1), and the London Mathematical Society with support from Heilbronn Institute for Mathematical Research (grant ECF-1920-57). 
In addition, the first and fourth authors gratefully acknowledge support from: 
the Danish National Research Foundation (grant DNRF156); 
the Independent Research Fund Denmark (grant 1026-00050B); and 
the Aarhus University Research Foundation (grant AUFF-F-2020-7-16). 
\end{acknowledgements}}


{\setstretch{1}\bibliographystyle{mybst}
\bibliography{references}}
\end{document}